\documentclass{article}
\usepackage{multicol,lipsum}
\usepackage{microtype}
\usepackage{bbm}

\usepackage{graphicx}
\usepackage{subfigure}
\usepackage{amsmath}
\usepackage{dcolumn}
\usepackage{amssymb}
\usepackage{amsthm,multirow}
\usepackage{booktabs} 
\newtheorem{theorem}{Theorem}

\newtheorem{lemma}{Lemma}
\usepackage{times}
\usepackage{microtype,caption}
\usepackage[colorinlistoftodos,textwidth=.6in]{todonotes}
\usepackage{dsfont}
\usepackage{algorithm2e}
\usepackage{mathtools}
\usepackage{hyperref}

\newcommand\blfootnote[1]{%
  \begingroup
  \renewcommand\thefootnote{}\footnote{#1}%
  \addtocounter{footnote}{-1}%
  \endgroup
}

\usepackage{arxiv}
\usepackage[bottom]{footmisc}
\usepackage[utf8]{inputenc} 
\usepackage[T1]{fontenc}    
\usepackage{hyperref}       
\usepackage{url}            
\usepackage{booktabs}       
\usepackage{amsfonts}       
\usepackage{nicefrac}       
\usepackage{microtype}      
\usepackage{lipsum}
\usepackage[xspace]{ellipsis}
\newtheorem{assumption}{Assumption}

\newcommand{\OTheta}[1]{\Theta\left(#1\right)}

\title{Making the Last Iterate of SGD \\ Information Theoretically Optimal}

\author{
   Prateek Jain \\
  Microsoft Research\\
  Bengaluru, India \\
  \texttt{prajain@microsoft.com} \\
  \And
  Dheeraj Nagaraj\thanks{} \\
  Department of Electrical Engineering and Computer Science\\
  Massachusetts Institute of Technology\\
  Cambridge, USA 02139 \\
  \texttt{dheeraj@mit.edu} \\
   \And
  Praneeth Netrapalli \\
  Microsoft Research\\
  Bengaluru, India \\
  \texttt{praneeth@microsoft.com} \\
}

\begin{document}
\maketitle

\begin{abstract}
Stochastic gradient descent (SGD) is one of the most widely used algorithms for large scale optimization problems. While classical theoretical analysis of SGD for convex problems studies (suffix) \emph{averages} of iterates and obtains information theoretically optimal bounds on suboptimality, the \emph{last point} of SGD is, by far, the most preferred choice in practice. The best known results for last point of SGD~\cite{shamir2013stochastic} however, are suboptimal compared to information theoretic lower bounds by a $\log T$ factor, where $T$ is the number of iterations.~\cite{harvey2018tight} shows that in fact, this additional $\log T$ factor is tight for standard step size sequences of $\OTheta{\frac{1}{\sqrt{t}}}$ and $\OTheta{\frac{1}{t}}$ for non-strongly convex and strongly convex settings, respectively. Similarly, even for subgradient descent (GD) when applied to non-smooth, convex functions, the best known step-size sequences still lead to $O(\log T)$-suboptimal convergence rates (on the final iterate). The main contribution of this work is to design new step size sequences that enjoy information theoretically optimal bounds on the suboptimality of \emph{last point} of SGD as well as GD. We achieve this by designing a modification scheme, that converts one sequence of step sizes to another so that the last point of SGD/GD with modified sequence has the same suboptimality guarantees as the average of SGD/GD with original sequence. We also show that our result holds with high-probability. We validate our results through simulations which demonstrate that the new step size sequence indeed improves the final iterate significantly compared to the standard step size sequences.
\end{abstract}

\blfootnote{Accepted for presentation at the Conference on Learning Theory (COLT) 2019}

\section{Introduction}
\label{sec:intro}
Stochastic Gradient Descent (SGD) is one of the most popular algorithms for solving large-scale empirical risk minimization (ERM) problems \cite{lecun2015deep,shalev2011pegasos,akiba2017extremely}. The algorithm updates the iterates using stochastic gradients obtained by sampling data points uniformly at random. The algorithm has been studied for several decades \cite{bubeck2015convex} but there are still significant gaps between {\em practical implementations} and theoretical analyses. In particular, the standard analyses hold only for some kind of average of iterates, but most practitioners  just use the final iterate of SGD. So, \cite{shamir2012open} asked the natural question of whether the final iterate of SGD, as opposed to average of iterates, is provably good. It was partly answered in \cite{shamir2013stochastic} which gave sub-optimality bound for the last point of SGD but the obtained sub-optimality rates are $O(\log T)$ worse than the information theoretically optimal rates; $T$ is the number of iterations. 

\cite{harvey2018tight} showed that the above result is tight for the standard step-size sequence used by most existing theoretical results. The extra logarithmic factor is not due to the stochastic nature of SGD. In fact, even for {\em subgradient descent} (GD) when applied to general non-smooth, convex functions, the last point's convergence rates are sub-optimal by $O(\log T)$ factor. 

So, this work addresses the following two fundamental questions: \\
{\em 
``Does there exist a step-size sequence for which the last point of SGD 
when applied to general convex functions as well as to strongly-convex
functions has optimal error (sub-optimality) rate?'', and,\\
}
{\em 
``Does there exist a step-size sequence for which the last point of GD 
when applied to general non-smooth convex functions has optimal error 
(sub-optimality) rate?''\\
}
In this paper, we answer both the questions in the affirmative.
That is, we provide novel step size sequences and show that the final iterate of SGD run with these step size sequences has the information theoretically optimal error (suboptimality) rate. In particular, for general non-smooth convex functions, our results ensure an error rate of $O(\frac{1}{\sqrt{T}})$ and for strongly-convex functions, the error rate is $O(\frac{1}{T})$. We also present high-probablity versions, i.e., we show that with probability at least $1-\delta$, the suboptimality is $O\left(\sqrt{\tfrac{\log{\tfrac{1}{\delta}}}{T}}\right)$ and  $O\left(\tfrac{\log{\tfrac{1}{\delta}}}{T}\right)$ respectively (see Theorems~\ref{thm:main_theorem_1} and~\ref{thm:main_theorem_2}). For GD, we show that a similarly modified step-size sequence leads to suboptimality of $O(\frac{1}{{T}})$ and $O(\frac{1}{\sqrt{T}})$ for non-smooth convex functions, with and with out strong convexity respectively, which is optimal.

In general, SGD takes the iterates near the optimum value but since the objective isn't smooth near the optimizer $x^{*}$, the gradients don't become small even when the points are close to $x^{*}$. Standard step sizes don't decay appreciably with time to ensure fast enough convergence to $x^*$. Therefore the iterates $x_t$, after going close to $x^{*}$, start oscillating around it without actually approaching it (See Section~\ref{sec:simulations} for concrete examples). Our new step sizes, given in Section~\ref{subsubsec:general_step_modification} ensure that the step sizes decay fast enough after a certain point, making the iterates go closer to the optimum $x^{*}$. The exact mode of this decay ensures that the last iterate approaches the optimum at the information theoretic rate.

Our results utilize a general step size modification scheme which ensures that the upper bounds for the average function value with the original step sizes gets transferred to the last iterate when the modified step sizes are used (see Theorems~\ref{thm:general_theorem_expectation} and~\ref{thm:general_theorem_high_probability}).  A key technical contribution of the paper is the proof of Theorem~\ref{thm:main_theorem_2} that  constructs a sequence of averaging schemes which are `good' with high probability such that the last averaging scheme consists only of the last iterate and hence lets us conclude that the last iterate is `good' with high probability. 

Our new step-size sequence requires that the number of iterations or horizon $T$ is known apriori. In contrast, standard step-size sequences do not require $T$ apriori, and hence guarantee any-time results. Information about $T$ apriori helps us in ensuring that we do not drop step-size too early; only after we are close to the optimum, does the step size drop rapidly. In fact, we conjecture that in absence of apriori information about $T$, {\em no step-size} sequence can ensure the information theoretically optimal error rates for final iterate of SGD. 
As a step towards proving this, we show that in the case of strongly convex objectives, any choice of step sizes with infinite horizon (i.e, without the knowledge of total number of iterations) is either suboptimal almost surely or suboptimal in expectation for infinitely many points. We show this in Theorem~\ref{thm:lower_bounds}.



\textbf{Related Work}: Averaging was used first in the stochastic approximation setting by \cite{polyak1992acceleration} to show optimal rates of convergence. Gradient Descent type methods have been shown to achieve information theoretically optimal error rates in the convex and strongly convex settings when averaging of iterates is used (\cite{nemirovsky1983problem},\cite{zinkevich2003online},\cite{cesa2004generalization}, \cite{kakade2009generalization}, Epoch GD in \cite{hazan2014beyond} , SGD \cite{rakhlin2012making} and \cite{lacoste2012simpler}). The question of the last iterate was first considered in \cite{shamir2013stochastic} and 
it gives a bound of $O(\frac{\log{T}}{\sqrt{T}})$ and $O(\frac{\log{T}}{T})$ in expectation for the general case and strongly convex case respectively. \cite{harvey2018tight} show matching high probability bounds and show that for the standard step sizes ($O\left(\frac{1}{\sqrt{t}}\right)$ in the general case and $O\left(\frac{1}{t}\right)$ in the strongly convex case), the logarithmic-suboptimal bounds are tight.

\textbf{Organization}:
The setting and main results are presented in Section~\ref{sec:main_results}. In particular, Section~\ref{subsubsec:general_step_modification} describes the general step size modification considered and states key results regarding this modification and the lower bound is presented in Section~\ref{sec:lower_bounds_stat}. Key technical ideas are developed in Section~\ref{sec:proofs} and the main theorems are proved. We present some experimental results in Section~\ref{sec:simulations} and conclude in Section~\ref{sec:disc}. Skipped proofs of technical lemmas are given in the appendix.


\section{Problem Setup and Main Results}
\label{sec:main_results}
Consider the following optimization problem:
\begin{equation}
\min_{x\in \mathcal{W}} F(x),
\end{equation}
where objective function $F :\mathbb{R}^d \to \mathbb{R}$ is a convex function and $\mathcal{W}\subset \mathbb{R}^d$ is  a closed convex set.  Let the global minimizer of $F(\cdot)$ be  $x^{*} \in \mathcal{W}$. We start the SGD algorithm at a point $x_1 \in \mathcal{W}$ and iteratively obtain estimates $x_t$ for the minimizer of $F(\cdot)$. We assume that at each time step, we have access to independent, unbiased estimate $\hat{g}_t$ to a subgradient $g_t \in \partial F$. That is, $\mathbb{E}[\hat{g}_t (x)]= g_t (x) \in \partial F(x)$ for every $x \in \mathcal{W}$ and $(\hat{g}_t-g_t)_{t=1}^{T}$ are independent. We pick step sizes $(\alpha_t)_{t=1}^{T} \geq 0$.  Let $\Pi_{\mathcal{W}}$ be the projection operator to the set $\mathcal{W}$. The SGD algorithm is given as follows:


\begin{algorithm}
\caption{Stochastic Gradient Descent}
\label{alg:net}
\KwIn{ total time $T$ and step sizes $\alpha_t$}
\KwOut{$x_T$}
\For{$t\leftarrow 1$ \KwTo $T$}{
$x_{t+1} \leftarrow \Pi_\mathcal{W}\left(x_{t} - \alpha_t  \hat{g}_{t}(x_{t})\right)\,.$
}
\end{algorithm}

Henceforth, we will retain the assumptions made above. Whenever we use $g_t(x)$, it is implied that $g_t (x) \in \partial F(x)$. Throughout the paper, we assume that $F$ is a Lipschitz continuous convex function. 
\begin{assumption}[Lipschitz Continuity]\label{as:lc}
	$F:\mathbb{R}^d \rightarrow \mathbb{R}$ is $G$-Lipschitz continuous convex function over closed convex set $\mathcal{W}$, i.e., $\|g(x)\| \leq G$ for every $x \in \mathcal{W}$ and every $g (x) \in \partial F(x)$. Furthermore, the stochastic gradients $\hat{g}$ satisfy: $\|\hat{g}(x)\|\leq G$ almost surely for every $x\in \mathcal{W}$. 
\end{assumption}
\begin{assumption}[Closed and bounded set]\label{as:dia}
	Diameter of closed convex set $\mathcal{W}$ is bounded by $D$, i.e., $\mathrm{diam}(\mathcal{W}) \leq D$. 
\end{assumption}
\begin{assumption}[Strong convexity]\label{as:sc}
Let $\lambda > 0$. A convex function $F$ is said to be $\lambda$ strongly convex over $\mathcal{W}$ iff
$F(y) \geq F(x) + \langle \nabla F(x), y-x\rangle + \frac{\lambda}{2} \|y-x\|^2 \; \forall \; x,y\in \mathcal{W}$.
\end{assumption}

\noindent{\bf Step size sequence for general convex functions}: we first define, \begin{equation}\label{eq:def_Ti}k  :=\inf\{i : T\cdot 2^{-i} \leq 1\},\ \  T_i := T - \lceil T\cdot 2^{-i}\rceil,\ 0\leq i\leq k, \  \ \mbox{ and } T_{k+1} := T.\end{equation}Clearly, $0 = T_0 < T_1 <\dots < T_k  = T-1 < T_{k+1} = T $. We note in particular that $T_1 \approx \frac{T}{2}$. Let $C >0$ be arbitrary. Then, we choose the step size $\alpha_t$ as follows:
\begin{equation}
\alpha_t = \frac{C\cdot 2^{-i}}{\sqrt{T}} \quad \text{when } T_i <t \leq T_{i+1}, 0\leq i \leq k.
\label{eq:weak_step_size}
\end{equation} 
The theorem below provides suboptimality guarantee for the SGD algorithm  with the step-size sequence mentioned above. 

\begin{theorem}[SGD/GD Last Point for General Convex Functions]
Let Assumptions~\ref{as:lc} and \ref{as:dia} hold. Given $T\geq 4$, let $x_1,\dots, x_T$ be the iterates of SGD (Algorithm~\ref{alg:net}) with step size $\alpha_t$ as defined in Equation~\eqref{eq:weak_step_size}. Then, the following holds for all $T\geq 4$: 
$$\mathbb{E}[F(x_T)] \leq F(x^*) + \frac{4D^2}{C\sqrt{T}} + \frac{11G^2C}{\sqrt{T}} \,.$$ 
In particular, if we choose $C = \frac{D}{G}$, we have: $\mathbb{E}F(x_T) \leq F(x^*) +\frac{15GD}{\sqrt{T}} \,.$ Furthermore, the following holds w.p. $\geq 1-\delta$ for any $0 < \delta < \frac{1}{e}$: 
$$F(x_T) = F(x^*) + O\left(\frac{D^2}{C\sqrt{T}} + \frac{CG^2}{\sqrt{T}} \log\left(\tfrac{1}{\delta}\right)\right)\leq F(x^*) + O\left(DG\sqrt{\tfrac{\log{\tfrac{1}{\delta}}}{T}}\right).$$
Finally, under the same assumptions, GD update $(x_{t+1}=\Pi_{\mathcal{W}}(x_t-\alpha_t \nabla F(x_t)))$ with the same step-size sequence given in \eqref{eq:weak_step_size} also ensures the following after $T$ iterations: 
$$F(x_T) \leq F(x^*) + \frac{4D^2}{C\sqrt{T}} + \frac{11G^2 C}{\sqrt{T}} \,.$$ 
\label{thm:main_theorem_1}
\end{theorem}
We will prove this theorem in Section~\ref{sec:proofs} after developing some general ideas.

{\bf Remarks:} (1) Note that the bounds on sub-optimality (for SGD and GD) are information theoretically optimal up to constants.\\
(2) Our result on the expected sub-optimality improves upon that of~\cite{shamir2013stochastic} by a multiplicative $\log T$ factor and our result on the high probability sub-optimality improves upon~\cite{harvey2018tight} by a multiplicative factor of $\log T \sqrt{\log \frac{1}{\delta}}$. On the other hand, our step-size sequence requires apriori knowledge of $T$. We conjecture that for any-time algorithm (i.e., without apriori knowledge of $T$) expected error rate of $\frac{GD\log T}{\sqrt{T}}$ is information theoretically optimal. \\
(3) The rate obtained above for last point of GD (in the deterministic setting) is also optimal in the gradient oracle model and to the best of our knowledge, is the first such result for last point of GD.

{\bf Step size sequence for strongly-convex functions}: Let $F(\cdot)$ be $\lambda$ strongly convex (Assumption~\ref{as:sc}). Let $k  :=\inf\{i : T\cdot 2^{-i} \leq 1\}$. We pick $\alpha_{t}$ as follows:
\begin{equation}
\alpha_t = 2^{-i}\frac{1}{\lambda t}, \quad \forall\ T_i <t \leq T_{i+1}, 0\leq i \leq k.
\label{eq:strong_step_size}
\end{equation}
We now present our result for last point of SGD with strong-convexity assumption. 
\begin{theorem}[SGD Last Point for Strongly Convex Functions]
\label{thm:main_theorem_2}
Let $F$ satisfy Assumptions~\ref{as:lc} and~\ref{as:sc}. Then the following holds for the $T$-th iterate of the SGD algorithm (Algorithm~\ref{alg:net}) when run with the step size sequence given in Equation~\eqref{eq:strong_step_size}: 
$$\mathbb{E}[F(x_T)] \leq F(x^*) + \frac{130G^2}{\lambda T}.$$
Furthermore, the following holds for all $0<  \delta \leq 1/e$ with probability at least $1-\delta$: 
$$\mathbb{E}[F(x_T)] = F(x^*) + O\left(\frac{G^2\log(\tfrac{1}{\delta})}{\lambda T}\right).$$
Under the same assumptions, GD update $(x_{t+1}=\Pi_{\mathcal{W}}(x_t-\alpha_t \nabla F(x_t)))$ with the same step-size sequence given in \eqref{eq:strong_step_size} also ensures the following after $T$ iterations: 
$$F(x_T) \leq F(x^*) + \frac{130G^2 }{{\lambda T}} \,.$$ 
\end{theorem}
Here again, we note that the result is information theoretically optimal up to $\log(1/\delta)$ factor.

\subsection{General Step Size Modification}
\label{subsubsec:general_step_modification}
Theorems~\ref{thm:main_theorem_1} and ~\ref{thm:main_theorem_2} are consequences of our general results on step size modification that we present below. Consider SGD step size sequence $(\gamma_t)_{t=1}^{T}$. We obtain modified step size sequence $(\alpha_t)_{t=1}^{T}$ as follows:
\begin{equation}
\alpha_t := 2^{-i}\gamma_t \quad \forall\ T_i < t \leq T_{i+1} \text{ and } 0\leq i \leq k.
\label{eq:modified_step_size}
\end{equation}

Under certain mild conditions, we will show that the last iterate of SGD with step size $\alpha_t$ is as good as the average iterate of SGD with step size $\gamma_t$. We make these notions precise below:

\begin{assumption}[Slowly Decreasing Step Size Sequence]\label{as:decay}
We call a step size sequence $(\gamma_t)$ `decreasing' if $\gamma_{t+1}\leq \gamma_t$. We say that step size sequence $\gamma_t$ has `at most polynomial decay' with decay constant $0<\beta \leq 1$ if $\gamma_{2t} \geq \beta\gamma_{t}$ for every $t \geq 1$.
\end{assumption}

We have the following general theorem:
\begin{theorem}
\label{thm:general_theorem_expectation}
Let $(\gamma_t)_{t=1}^T$ be a decreasing step size sequence with at most polynomial decay with decay constant $ 0<\beta \leq 1$. Let the iterates of SGD with step size $\gamma_t$ be $y_1, \dots, y_T$. Let $\alpha_t$ be the modification of $\gamma_t$ as defined in Equation~\eqref{eq:modified_step_size}. Let the iterates of SGD with step size $\alpha_t$ be $x_1, \dots, x_T$. Then, for all $T \geq 4$, we have:  
$$\mathbb{E}[F(x_T)] \leq  5G^2\gamma_T\left(\frac{1}{\beta^2} + \frac{1}{\beta^4}\right) + \inf_{ \lceil \frac{T}{4}\rceil \leq t \leq T_1}\mathbb{E}[F(y_t)]\,.$$
\end{theorem}

We also give a high probability version of Theorem~\ref{thm:general_theorem_expectation}.
\begin{theorem}
\label{thm:general_theorem_high_probability}
Let $T \geq 4$. Let $q^{(0)}$ be any arbitrary fixed probability distribution over the set $\{\lceil\tfrac{T}{4}\rceil,\dots, T_1\}$.  With probability atleast $1-\frac{\delta}{2}$, we have: 
\begin{align*}
F(x_T) &\leq   \gamma_T \cdot G^2\left(120\log(\tfrac{1}{\delta}) + 400\right) \cdot \left(\frac{1}{\beta^2} + \frac{1}{\beta^4}\right) + \sum_{s = \lceil \tfrac{T}{4}\rceil}^{T_{1}} q^{(0)}(s)F(y_s).
\end{align*}
\end{theorem}
That is, the above theorems show that compared to any weighted average of function values of iterates in the $[T/4,T_1]$ iterations, the error is not significantly larger if $\beta$ is reasonably large and $\gamma_T$ is small. Now, using standard analysis, we can ensure small average function value for iterates in $[T/4,T_1]$ iterations. Small value of $\gamma_T$ and bound on $\beta$ hold trivially for standard step-size sequences.

See Section~\ref{sec:proofs} for detailed proofs of the above theorems. We first develop general technique and prove key lemmas in the next section, and then present proofs for all the theorems. 

\subsection{Lower Bounds}\label{sec:lower_bounds_stat}
The step size modification procedure described above assumed the knowledge of the last iterate $T$ (this is not a setback in practice). We study the case of infinite horizon SGD. In this section we state our bounds on the last iterate of `any time' (infinite horizon) SGD in the case of strongly convex objectives. We will first introduce the notion of suboptimality that we consider. In particular, we look at two kinds of `bad performance' in infinite horizon SGD for non-smooth strongly convex optimization. Consider any infinite step size sequence $\gamma_t$.

\begin{enumerate}
\item The sequence $\gamma_t$ is said to be `bad in expectation' if for an objective $F$ satisfying assumptions \ref{as:lc},\ref{as:dia} and \ref{as:sc}, some choice of subgradient oracle, and SGD iterates $(x_t)_{t \in \mathbb{N}}$ with step size $\gamma_t$, there is a fixed subsequence $\{t_k\}_{k \in \mathbb{N}}$ such that $\lim_{k\to \infty} t_k \mathbb{E}[F(x_{t_k})- F(x^*)] = \infty $.
\item The sequence $\gamma_t$ is said to be `bad almost surely' if for an objective $F$ satisfying assumptions \ref{as:lc},\ref{as:dia} and \ref{as:sc}, some choice of subgradient oracle, and SGD iterates $(x_t)_{t \in \mathbb{N}}$ with step size $\gamma_t$, with probability $1$ there exists a random infinite sequence of times $\{t_k\}$ such that $\lim_{k\to \infty} t_k [F(x_{t_k})- F(x^*)] = \infty$
\end{enumerate}
We give a `no free lunch' theorem: that is we show that infinite horizon step-size sequence for non-smooth strongly convex optimization is either `bad in expectation' or `bad almost surely'. More precisely, we will show that if any infinite horizon SGD is good in 	`expectation' for every $t$ for every strongly convex function, then it is `bad almost surely' for some function $F$. 

\begin{theorem}\label{thm:lower_bounds}
Consider infinite horizon SGD with step size $\gamma_t$ such that assumptions~\ref{as:lc}~\ref{as:dia} and ~\ref{as:sc} hold for the objective function. Then, for any choice of $\gamma_t >0$, the algorithm is either bad in expectation or bad almost surely. 
\end{theorem}
We give the proof in Section~\ref{sec:lower_bounds}.

\section{Technical Ideas and Proofs}
\label{sec:proofs}
 Recall the definition of $T_i$ from Section~\ref{sec:main_results}. The rough idea behind the proof is as follows: we will find a `good point' in the range $\left[\lceil \tfrac{T}{4}\rceil, T_1\right]$ and then show that this implies that there is a `good point' between $T_1=T/2$ and $T_2\approx 3T/4$ and so on, until we conclude that $x_T$ is a good point. 

To this end, we first provide a key lemma that bounds the total weighted deviation  of SGD iterates from a given iterate $x_{t_0}$ (in terms of function value), i.e., it intuitively shows that once we find an iterate with small function value, the remaining iterates cannot deviate from it significantly. The lemma uses a trick that was first used in~\cite{zhang2004solving} and then also in~\cite{shamir2013stochastic}. 
\begin{lemma}
	Let $x_1,\dots, x_T$ be the output of SGD algorithm (Algorithm~\ref{alg:net}) with step size sequence $\alpha_t$ defined by \eqref{eq:weak_step_size}. Then, given any $ 1<t_0 < t_1 \leq T$, $$\sum_{t=t_0}^{t_1}2\alpha_t\mathbb{E}\left[F(x_t)-F(x_{t_0})\right] \leq \sum_{t=t_0}^{t_1}G^2\alpha_t^2. $$
\label{lem:look_ahead}
\end{lemma}
\begin{proof}
By convexity of $\mathcal{W}$, we have:
\begin{align*}
\|x_{t+1} -x_{t_0}\| &= \|\Pi_{\mathcal{W}}\left(x_t  - \alpha_t \hat{g}_t(x_{t})\right) - x_{t_0}\| \leq \|x_t  - \alpha_t \hat{g}_t(x_{t}) - x_{t_0}\|
\end{align*}
Taking squares and expanding on both sides,
\begin{align*}
\|x_{t+1}-x_{t_0}\|^2 &\leq \|x_t - x_{t_0}\|^2 + \alpha_t^2\|\hat{g}_t(x_t)\|^2 - 2\alpha_t \langle \hat{g}_t(x_{t}),x_{t}-x_{t_0}\rangle 
\end{align*}
Taking expectation on both sides, and realizing that $\hat{g}_t$ is independent of $x_t$ and $x_{t_0}$, we conclude,
\begin{align*}
\mathbb{E}\left[\|x_{t+1}-x_{t_0}\|^2\right] &\leq \mathbb{E}\left[\|x_{t}-x_{t_0}\|^2\right] +\alpha_t^2 G^2 -2\alpha_t \mathbb{E}\langle g_t,x_t-x_{t_0}\rangle  
\end{align*}

Here we have used the fact that $\mathbb{E}\left[\hat{g}_t(x_t)|x_t,x_{t_0}\right] = g_t(x_t)$. Using convexity, $\langle g_t,x_t-x_{t_0}\rangle$ is lower bounded by $F(x_t) - F(x_{t_0})$. We conclude that:
\begin{align*}
\mathbb{E}\left[\|x_{t+1}-x_{t_0}\|^2\right] &\leq \mathbb{E}\left[\|x_{t}-x_{t_0}\|^2\right]+\alpha_t^2 G^2-2\alpha_t \mathbb{E}\left[ F(x_t)-F(x_{t_0})\right]  
\end{align*}
The result now follows by summing the above term from $t = t_0$ to $t = t_1$.
\end{proof}

We now provide a high probability version of Lemma~\ref{lem:look_ahead}. To this end, we construct an exponential super-martingale that when combined with a  Chernoff bound leads to exponential concentration bound. The method used is somewhat similar to the one used in \cite{harvey2018tight}, but our technique is specifically for Lemma~\ref{lem:look_ahead} and is more concise. 

For simplicity of exposition, we first define a few key quantities. Let $1 < t_0 < t_{1} \leq T$ and $r = t_1 - t_0 + 1$. We define the sequence $L_t$ as follows: 
 for $ t_0 \leq t \leq t_{1}$ as follows: \begin{equation}\label{eq:def_L}
L_{t_1} = \frac{1}{e\cdot r},\ t\in [t_0,t_1],\ \ \ L_{t-1} = L_t + L_t^2,\  t_0 \leq t -1 < t_{1}.
\end{equation}
Using Lemma ~\ref{lem:lambda_growth_bound}, $\frac{1}{e\cdot r} \leq L_{t} \leq \frac{1}{r}$. Now, for any $l$ such that $t_0 \leq l \leq t_1$, we define the following random variables :{\footnotesize 
\begin{equation}\label{eq:arv}A(l,t_1) := \sum_{t=l}^{t_1} L_t\left[2\alpha_t (F(x_{t}) - F(x_{l})) - \alpha_t^2G^2 \right],\ A^{*}(t_0,t_1) := \sum_{t=t_0}^{t_1} L_t\left[2\alpha_t (F(x_{t}) - F(x^*)) - \alpha_t^2G^2 \right].\end{equation}}
We note the difference between $A^*(t_0,t_1)$ and $A(l,t_1)$: $A(l,t_1)$ considers suboptimality with respect to $x_l$ whereas $A^{*}(t_0,t_1)$ considers the suboptimality with respect to the optimizer $x^*$.

\begin{lemma}
\label{lem:look_ahead_high_probability}
Let $A$ and $A^*$ be as defined by \eqref{eq:arv}. 
Let $p(t_0),\dots,p(t_1)$ be any probability distribution over $\{t_0,\dots,t_1\}$. We let $(p.A)(t_0,t_1) := \sum_{l = t_0}^{t_1} p(l) A(l,t_1)$. Also, let $\alpha_t$ be a decreasing step size sequence. Then,
$$\mathbb{P}\left[(p.A)(t_0,t_1) > \eta\right] \leq \exp\left(-\tfrac{\eta}{8\alpha^2_{t_0}G^2}\right)\,.$$
Additionally, if $diam(\mathcal{W}) \leq D$ almost surely, we have:

$$\mathbb{P}\left[A^{*}(t_0,t_1) > \eta\right] \leq \exp\left(\tfrac{2D^2L_{t_0}}{8\alpha_{t_0}^2G^2}\right)\exp\left(\tfrac{- \eta}{8\alpha^2_{t_0}G^2}\right).$$
\end{lemma}
\begin{lemma}
Let $\Gamma >0$ be fixed. Let $\lambda_0 = \frac{1}{re\Gamma}$, $\lambda_1 = \lambda_0 + \Gamma\lambda_0^2$, $\dots$, $\lambda_{i+1} = \lambda_{i} + \Gamma\lambda^2_{i}$.  Then, for every $i \leq r$, $\lambda_i \leq (1+\tfrac{1}{r})^i\lambda_0$
\label{lem:lambda_growth_bound}
\end{lemma}
See Section~\ref{sec:proofs_appendix} for proofs of the above given lemmata. We also require the following technical lemma: 
\begin{lemma}
	\label{lem:division_length_bound}
	Let $T_i$ be as defined in Section~\ref{sec:main_results}. Then, for  all $0\leq i \leq k-1$: 
	$$4\left(T_{i+2} - T_{i+1}\right) \geq T_{i+1} - T_i.$$
\end{lemma}
\begin{proof}
	Lemma follows from the fact that $2\lceil a \rceil -1 \leq \lceil 2a \rceil \leq 2\lceil a \rceil$. 
\end{proof}

\subsection{Step Size Modification}
\label{subsec:step_size_modification_properties}
Henceforth, we will assume that $\gamma_t$ is a decreasing step size sequence with at most polynomial decay (decay constant being $\beta$). We let $\alpha_t$ be the modification of $\gamma_t$ as defined in Equation~\ref{eq:modified_step_size}. Let, \begin{equation}\label{eq:def_tau}\tau_i := \arg\inf_{ T_i <t \leq T_{i+1}} \mathbb{E}[F(x_t)],\ i\in [k+1],\ \  \mbox{ and }\tau_0 := \arg\inf_{\lceil \frac{T}{4}\rceil \leq t\leq T_1 }\mathbb{E}[F(x_t)].\end{equation} 
Note that $\tau_{k+1} = T$. We note that $\tau_i$ are completely deterministic and only used as part of the proof. The ability to compute $\tau_i$ is not necessary.
\begin{lemma}
\label{lem:good_point_transfer}
Let $x_t$'s be iterates of SGD (Algorithm~\ref{alg:net}) with modified step size sequence  $\alpha_t$ of $\gamma_t$ defined in \eqref{eq:modified_step_size}; $\gamma_t$ sequence satisfies Assumption~\ref{as:decay}. Let $T_i,\ k$ be as defined by \eqref{eq:def_Ti}, and $\tau_i$, $0\leq i\leq k+1$ be as defined in \eqref{eq:def_tau}. Also, let $T\geq 4$. Then, the following holds  for all $i\in [k]$: 
$$\mathbb{E}[F(x_{\tau_{i+1}}) - F(x_{\tau_{i}})] \leq \frac{5G^2\gamma_T}{\beta^2}2^{-i}, \ \ \mathbb{E}[F(x_{\tau_{1}}) - F(x_{\tau_{0}})] \leq \frac{5G^2\gamma_T}{\beta^4}.$$
\end{lemma}
\begin{proof}
 We first consider $i\geq 1$. 
 If $\mathbb{E}[F(x_{\tau_{i+1}})] \leq \mathbb{E}[F(x_{\tau_i})]$, the proof is done. Else, using Lemma~\ref{lem:look_ahead} with $t_0 =\tau_i$ and $t_1 = T_{i+2}$, and the fact that $\alpha_t$ is a decreasing sequence, we get:
 \begin{align}
\frac{\sum_{t=\tau_i}^{T_{i+2}}2\alpha_t\mathbb{E}\left[F(x_t)-F(x_{\tau_i})\right] }{T_{i+2}-\tau_i + 1}&\leq \frac{\sum_{t=\tau_i}^{T_{i+2}}G^2\alpha_t^2 }{T_{i+2}-\tau_i+1}\leq G^2 \alpha_{T_i+1}^2. \label{eq:first_equation}
 \end{align}
 
 By definition of $\tau_i$, $\mathbb{E}[F(x_{\tau_i})]\leq \mathbb{E}[F(x_t)]$ whenever $T_i <  t \leq T_{i+1}$. Hence, 
 \begin{align}
 G^2 2^{-2i}\gamma_{T_i+1}^2 &= G^2 \alpha_{T_i+1}^2 \geq \tfrac{\sum_{t=\tau_i}^{T_{i+2}}2\alpha_t\mathbb{E}\left[F(x_t)-F(x_{\tau_i})\right] }{T_{i+2}-\tau_i + 1} \geq \tfrac{\sum_{t=T_{i+1}+1}^{T_{i+2}}2\alpha_t\mathbb{E}\left[F(x_t)-F(x_{\tau_i})\right] }{T_{i+2} - \tau_i +1 },
\end{align} 
where the first equality follows from the definition of $\alpha_t$ in \eqref{eq:modified_step_size}, first inequality follows from Equation~\eqref{eq:first_equation}, and the final inequality follows from the fact that $\mathbb{E}[F(x_{t}) - F(x_{\tau_i})] \geq 0$ when $T_i<t \leq T_{i+1}$ (see definition of $\tau_i$ in \eqref{eq:def_tau}). 

Now, by using the above inequality with the assumption $ \mathbb{E}[F(x_{\tau_{i+1}})]\geq \mathbb{E}[F(x_{\tau_i})]$, and the fact that $T_{i+2} - T_i \geq T_{i+2} - \tau_i +1$, we have: 
\begin{align}
G^2 2^{-2i}\gamma_{T_i+1}^2 
 &\geq 2\alpha_{T_{i+2}}\tfrac{T_{i+2} - T_{i+1}}{T_{i+2}-T_i}\mathbb{E}\left[F(x_{\tau_{i+1}})-F(x_{\tau_i})\right] \stackrel{\zeta_1}{\geq} \tfrac{2\alpha_{T_{i+2}}}{5} \mathbb{E}\left[F(x_{\tau_{i+1}})-F(x_{\tau_i})\right] \nonumber \\
 &=\tfrac{2^{-i}\gamma_{T_{i+2}}}{5} \mathbb{E}\left[F(x_{\tau_{i+1}})-F(x_{\tau_i})\right] \geq \tfrac{2^{-i}\beta \gamma_{T_{i+1}}}{5} \mathbb{E}\left[F(x_{\tau_{i+1}})-F(x_{\tau_i})\right],
 \label{eq:second_equation}
 \end{align}
 where $\zeta_1$ follows from Lemma~\ref{lem:division_length_bound}. The equality follows from definition of $\alpha_t$ and the last inequality follows from the $\beta$-slowly decaying assumption for $\gamma_t$ (Assumption~\ref{as:decay}).
 That is we obtain the result for the case $i\geq 1$. The proof for the case when $i=0$ follows with minor modifications to the arguments given above.
 \end{proof}

We now present a high probability version of Lemma~\ref{lem:good_point_transfer}. 
\begin{lemma}
\label{lem:good_point_transfer_high_probability}
Consider the setting of Lemma~\ref{lem:good_point_transfer}. Let $0\leq i\leq k$ and define $t_0=T_i+1$ for $1\leq i$ and $t_0= \lceil \tfrac{T}{4}\rceil$ for $i=0$. Let $q^{(i)}$ be any probability distribution over $\{t_0,\dots, T_{i+1}\}$. Let $p^{i+1}(t) := \frac{L_t\alpha_t}{\sum_{s=T_{i+1}+1}^{T_{i+2}}L_s\alpha_s}$, where $t\in [T_{i+1}+1,T_{i+2}]$ and the sequence $(L_t)_{T_{i+1}+1}^{T_{i+2}}$ is defined by \eqref{eq:def_L}. Then, for any $\delta_i \in (0,1)$ and $i\in [1,k-1]$, the following holds with probability at least $1-\delta_i$:
\begin{align*}
\sum_{t = T_{i+1}+1}^{T_{i+2}} p^{(i+1)}(t)F(x_t) &\leq  \frac{G^2\gamma_T 2^{-i}}{\beta^2}\left(15+ 120\log{\tfrac{1}{\delta_i}}\right) + \sum_{s = T_{i}+1}^{T_{i+1}} q^{(i)}(s)F(x_s).
\end{align*}
For $i = 0$, the following holds with probability atleast $1-\delta_0$:
\begin{align*}
\sum_{t = T_1+1 }^{T_{2}} p^{(1)}(t)F(x_t) &\leq \frac{G^2\gamma_T 2^{-i}}{\beta^4}\left(15+ 120\log{\tfrac{1}{\delta_i}}\right)+ \sum_{s = \lceil \tfrac{T}{4}\rceil}^{T_{1}} q^{(0)}(s)F(x_s).
\end{align*}

\end{lemma}
\begin{proof}
We will only show the case $ 1 \leq i \leq k-1$. The $i=0$ case follows by a similar proof. For $ T_i \leq t \leq T_{i+1}$, we define $\Gamma(t)= \sum_{s = t+1}^{T_{i+2}}\alpha_s L_s$.
We let $\kappa$ be defined as follows over $\{T_i+1,\dots,T_{i+2}\}$:
\begin{multline}\label{eq:def_kappa}
\kappa(T_i+1) := \tfrac{\Gamma(T_{i+1})}{\Gamma(T_i+1)}\cdot q^{(i)}\left(T_i +1\right),\ \kappa(t) := \tfrac{\Gamma(T_{i+1})}{\Gamma(t)}q^{(i)}(t) + \tfrac{\alpha_t L_t\cdot \left(\sum_{s = T_i + 1}^{t-1} \kappa(s)\right)}{\Gamma(t)},\ t\in (T_i+1,T_{i+1}],\\ \ \kappa(t) := 0,\ \forall t\geq T_{i+1}.
\end{multline}
%
From Lemma~\ref{lem:sums_to_one}, we conclude that $\kappa$ is a probability distribution over $\{T_i+1,\dots,T_{i+2}\}$. From Lemma~\ref{lem:look_ahead_high_probability}, we conclude that with probability atleast $1-\delta_i$:

\begin{equation}
(\kappa.A)(t_0,t_1) \leq 8\alpha^2_{t_0}G^2\log{\tfrac{1}{\delta_i}}
\end{equation}
We will show that when this event happens, the inequality in the statement of the lemma holds. If $\sum_{t = T_{i+1}+1}^{T_{i+2}} p^{(i+1)}(t)F(x_t) \leq \sum_{s = T_{i}+1}^{T_{i+1}} q^{(i)}(s)F(x_s) $, then the statement of the lemma holds trivially. Now assume $\sum_{t = T_{i+1}+1}^{T_{i+2}} p^{(i+1)}(t)F(x_t) > \sum_{s = T_{i}+1}^{T_{i+1}} q^{(i)}(s)F(x_s) $. 
We use the fact that $\kappa$ is supported over $\{T_{i}+1,\dots,T_{i+1}\}$ and hence:

$$
(\kappa.A)(t_0,t_1) = \sum_{l=T_i +1}^{T_{i+1}}\sum_{t=l}^{T_{i+2}} \kappa(l)L_t\left[2\alpha_t (F(x_t) - F(x_l))-\alpha_t^2G^2\right]
$$
We exchange summation and collect the coefficients of the term $F(x_t)$ to conclude:
{\small 
\begin{multline*}
\hspace*{-5pt}(\kappa.A)(t_0,t_1) = \hspace*{-8pt}\sum_{t=T_{i+1}+1}^{T_{i+2}}\hspace*{-8pt}L_t\left(\alpha_t F(x_t)  - \alpha_t^2G^2\right)  - \sum_{s = T_{i}+1}^{T_{i+1}}\left( \alpha_s^2G^2 \sigma_sL_s  +2 F(x_s)\left( \sigma_s\alpha_sL_s - \kappa(s)\Gamma(s-1) \right)\right),
\end{multline*}}
where $\sigma(s) :=  \sum_{t=T_{i}+1}^{s}\kappa(s)$ (empty sum being $0$ by definition). By definition of $\sigma(s) = \kappa(s) + \sigma(s-1)$, $\Gamma(s-1) = \alpha_sL_s + \Gamma(s)$ and $\kappa(s) = \tfrac{\Gamma(T_{i+1})}{\Gamma(s)}q^{(i)}(s) + \tfrac{\alpha_s L_s}{\Gamma(s)}\sigma(s-1)$. Therefore, we conclude:
\begin{multline}\label{eq:main_equation_for_this_lemma}
(\kappa.A)(t_0,t_1) = \sum_{T_{i+1}+1}^{T_{i+2}}2\alpha_tL_t F(x_t)  - \sum_{T_{i+1}+1}^{T_{i+2}}\alpha_t^2G^2L_t  - \sum_{s = T_{i}+1}^{T_{i+1}} \alpha_s^2G^2L_s \sigma_s  \\\quad-\left( \sum_{T_{i+1}+1}^{T_{i+2}}2\alpha_tL_t \right)\left(\sum_{s = T_{i}+1}^{T_{i+1}}q^{(i)}(s)F(x_s)\right).
\end{multline}
We recall that $p^{(i+1)}(t) \cdot \left(\sum_{s=T_{i+1}+1}^{T_{i+2}}\alpha_sL_s\right) = \alpha_t L_t$ whenever $T_{i+1} < t\leq T_{i+2}$. The rest of the proof is similar to Equation~\eqref{eq:second_equation} in Lemma~\ref{lem:good_point_transfer}. We use the fact that $\alpha_t$ is the modification of $\gamma_t$, $\gamma_t$ has at most polynmial decay, $ \frac{1}{e(T_{i+2}-T_i)}\leq L_t \leq \frac{1}{T_{i+2} - T_i}$ and Lemma~\ref{lem:division_length_bound} in Equation~\eqref{eq:main_equation_for_this_lemma} to conclude the result.
\end{proof}

\begin{lemma}
\label{lem:sums_to_one}
Let $\kappa$ be as defined in \eqref{eq:def_kappa}. Then, $\kappa$ is a probability distribution over $\{T_i+1,\dots,T_{i+2}\}$.
\end{lemma}
The proof of this lemma is given in Section~\ref{sec:proofs_appendix}

\subsection{Proof of Theorem~\ref{thm:general_theorem_expectation}}
\begin{proof} 
Recall the definition of $\tau_i$ in \eqref{eq:def_tau}. Clearly, $\tau_{k+1}= T$. Summing the bounds in Lemma~\ref{lem:good_point_transfer} we conclude:
\begin{align*}
\mathbb{E}[F(x_{T})] &= \mathbb{E}[F(x_{\tau_{k+1}})]=  \mathbb{E}[F(x_{\tau_{0}})] + \sum_{i=0}^{k} \mathbb{E}[F(x_{\tau_{i+1}}) - F(x_{\tau_{i}})] \\
&\leq \mathbb{E}[F(x_{\tau_{0}})] +\frac{5G^2\gamma_T}{\beta^4} +\sum_{i=1}^{k}  \frac{5G^2\gamma_T}{\beta^2}2^{-i} \leq 5G^2\gamma_T\left(\frac{1}{\beta^2} + \frac{1}{\beta^4}\right)   + \inf_{ \lceil \frac{T}{4}\rceil \leq t \leq T_1}\mathbb{E}[F(x_t)].
\end{align*}
We conclude the result by noting that $x_t = y_t$ for all $ t \leq T_1$. 
\end{proof}

\subsection{Proof of Theorem~\ref{thm:general_theorem_high_probability}}
\begin{proof} 
This proof is similar to the proof of Theorem~\ref{thm:general_theorem_expectation}, but instead of Lemma~\ref{lem:good_point_transfer} we use Lemma~\ref{lem:good_point_transfer_high_probability}. In Lemma~\ref{lem:good_point_transfer_high_probability}, we pick $q^{(i)} = p^{(i)}$ for $1\leq i\leq k-1$ and we let $q^{(0)}$ be arbitrary. We let $\delta_i = \frac{\delta}{2^{i+2}}$. By union bound, the inequalities in the statement of Lemma~\ref{lem:good_point_transfer_high_probability} hold for all $0\leq i \leq k-1$ simultaneously with probabiliy atleast $1- \sum_{i=0}^{k-1}\delta_i \geq 1 - \frac{\delta}{2}$. Summing all these inequalities, we conclude:
\begin{align*}
\sum_{t = T_{k}+1}^{T_{k+1}} p^{(k)}(t)F(x_t) &\leq  \gamma_T G^2\left[120\log(\tfrac{1}{\delta}) + 400\right] \left[\frac{1}{\beta^2} + \frac{1}{\beta^4}\right]+\sum_{s = \lceil\tfrac{T}{4}\rceil }^{T_{1}} q^{(0)}(s)F(x_s).
\end{align*}

We note that the distribution $p^{(k)}$ has unit mass over the point $T_{k+1} = T$ and that $x_t = y_t$ when $t \leq T_1$ to conclude the result.
\end{proof}
\subsection{Proof of Theorem~\ref{thm:main_theorem_1}}
\begin{proof}
 We note that the step size defined in Equation~\eqref{eq:weak_step_size} is the modification of the standard step size $\gamma_t = \frac{C}{\sqrt{T}}$. Let $y_t$ be the output of SGD under the assumptions of the theorem when step size $\gamma_t$ is used. Using the fact that infimum is smaller than any weighted average, we have: 
\begin{multline*}
 \inf_{\lceil\tfrac{T}{4}\rceil \leq t\leq T_1} \mathbb{E}[F(y_t) - F(x^{*})]  \leq \tfrac{1}{T_1 - \lceil\tfrac{T}{4}\rceil +1} \sum_{t=\lceil\tfrac{T}{4}\rceil }^{T_1} \mathbb{E}[F(y_t) - F(x^{*})]
\\\leq  \tfrac{2}{T_1}\sum_{t=1 }^{T_1} \mathbb{E}[F(y_t) - F(x^{*})] \stackrel{\zeta_1}{\leq} \tfrac{2}{\sqrt{T_1}} \left[ \tfrac{D^2\sqrt{T_1}}{C} + CG^2\sqrt{T_1}\right] \leq \tfrac{4}{\sqrt{T}} \left[ \tfrac{D^2}{C} + CG^2\right],
\end{multline*}
where the second line follows from $T_1 \leq 2(T_1 - \lceil\frac{T}{4}\rceil +1)$. $\zeta_1$ follows from the standard analysis \cite{bubeck2015convex}.  Final inequality follows from the fact that $\frac{T}{4}\leq T_1 \leq \frac{T}{2}$. We note that $\gamma_t$ satisfies the conditions for Theorem~\ref{thm:general_theorem_expectation} with $\beta = 1$. We invoke Theorem~\ref{thm:general_theorem_expectation} to conclude the bound on expectation. The above proof in expectation also works for GD . We take $\hat{g}_t = \nabla F$ and SGD is the same as GD. Here each $x_t$ and $y_t$ is a deterministic point mass. Therefore, the expectation bound for the last iterate of SGD holds for the last iterate of GD.

We will now prove the high probability bound. Let $t_0 = \lceil \tfrac{T}{4}\rceil $, $t_1 = T_1$, and $\alpha_t = \gamma_t = \frac{C}{\sqrt{T}}$ for $t\in [t_0,t_1]$. Then using Lemma~\ref{lem:look_ahead_high_probability}, the following holds with probability atleast $1-\delta$: 
$$A^{*}(t_0,t_1) \leq 2D^2L_{t_0} + 8\alpha_{t_0}^2G^2 \log{\tfrac{1}{\delta}}.$$
Using $ \frac{1}{e(t_1-t_0+1)} \leq L_{t} \leq \frac{1}{t_1-t_0+1}$ and proceeding similarly as above, we have w.p. $\geq 1-\frac{\delta}{2}$,

\begin{align*}
\inf_{\lceil\tfrac{T}{4}\rceil \leq t \leq T_1} \mathbb{E}[F(y_t) - F(x^*)] \leq \tfrac{6D^2}{C\sqrt{T}} + \tfrac{5CG^2\log{\tfrac{2}{\delta}}}{\sqrt{T}}.
\end{align*}
Theorem now follows by using Theorem~\ref{thm:general_theorem_high_probability} with $\beta =1$, $q^{(0)}(t) = \frac{1}{T_1-\lceil\frac{T}{4}\rceil +1}$, and union bound. 
\end{proof}

\subsection{Proof of Theorem~\ref{thm:main_theorem_2}}
 \begin{proof}
 We note that the step size defined in Equation~\eqref{eq:strong_step_size} is the modification of the standard step size $\gamma_t = \frac{1}{\lambda t}$ used for strongly convex functions (see \cite{rakhlin2012making}). Let $y_1,\dots, y_T$ be the output of SGD when step size $\gamma_t$ is used. From Theorem 5 in \cite{rakhlin2012making}, we conclude that:
 \begin{align*}
\inf_{\lceil T/4\rceil \leq t \leq T_1} \mathbb{E}[F(y_t) - F(x^*)] 
\leq \tfrac{1}{T_1 -\lceil\tfrac{T}{4}\rceil +1}\sum_{t = \lceil T/4\rceil}^{T_1}\mathbb{E}[F(y_t) - F(x^*)]  \leq 30\tfrac{G^2}{\lambda T}.
\end{align*}
The expectation bound follows from using above equation with Theorem~\ref{thm:general_theorem_expectation} and noting that $\gamma_t$ satisfies the  required conditions with $\beta=2$. 
 We get high probability bounds by invoking  high probability bounds for suffix averaging from \cite{harvey2018tight}, i.e., w.p. at least $1-\frac{\delta}{2}$, 
\begin{align*}\inf_{\lceil\tfrac{T}{4}\rceil \leq t \leq T_1} F(y_t) - F(x^*) 
 &\leq O\left(\tfrac{G^2 \log\left(\tfrac{1}{\delta}\right)}{\lambda T}\right).
\end{align*}
The result now follows by using Theorem~\ref{thm:general_theorem_high_probability} with $\beta =2$ and $q^{(0)}(t) = \frac{1}{T_1 - \lceil\frac{T}{4}\rceil+1}$.
 \end{proof}

\section{Experiments}
\label{sec:simulations}
We now empirically compare SGD last point with our step-size sequence (Our Method) with the standard steps size sequence (Standard) as well as the averaged iterates of SGD (Averaged). We apply these methods on two non-smooth problems: a) Lasso regression, b) linear SVM training. 

\noindent{\bf Lasso Regression:} We consider gradient descent for $F(x) = \frac{1}{n}\sum_{i=1}^n \|\langle a_i,x \rangle - b_i\|^2 + \lambda \|x\|_1$ for $x \in \mathbb{R}^d$.  Here $a_i \sim \mathcal{N}(0,I_d)$ and $b_i = \langle a_i, x^*\rangle + z_i$ for some $s$ sparse vector $x^*$ and $z_i \sim \mathcal{N}(0,\sigma^2)$. $a_i$ and $z_i$ are all independent. We use the step sizes of $\gamma_t = \frac{C}{\sqrt{T}}$ and let $\alpha_t$ be the modification of $\gamma_t$ as given in Section~\ref{sec:main_results} for total $T$ iterations. 
\begin{figure}[!htb]
\centering
		\begin{tabular}{ccc}
		\includegraphics[width=.32\textwidth]{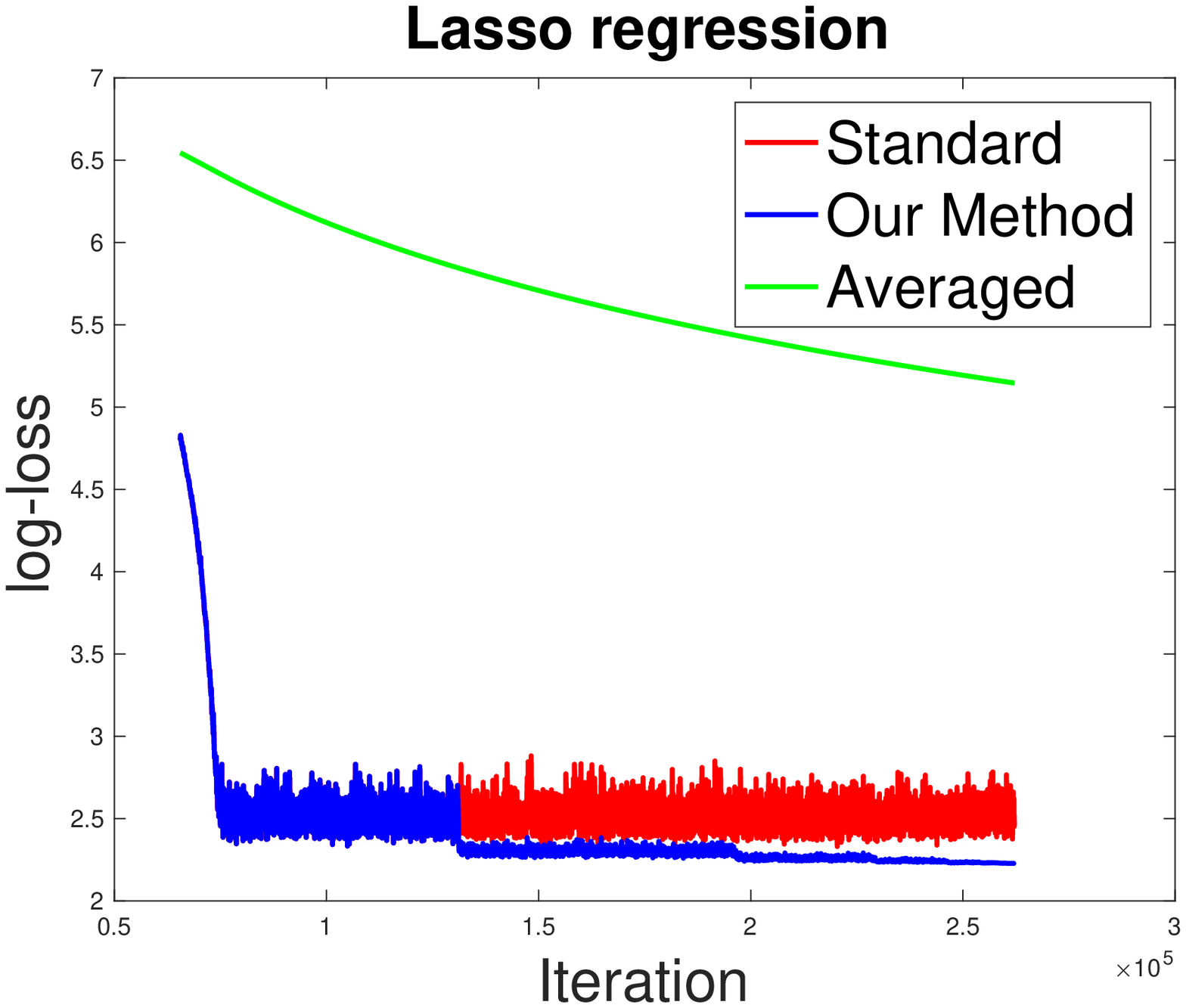}&
		\includegraphics[width=.32\textwidth]{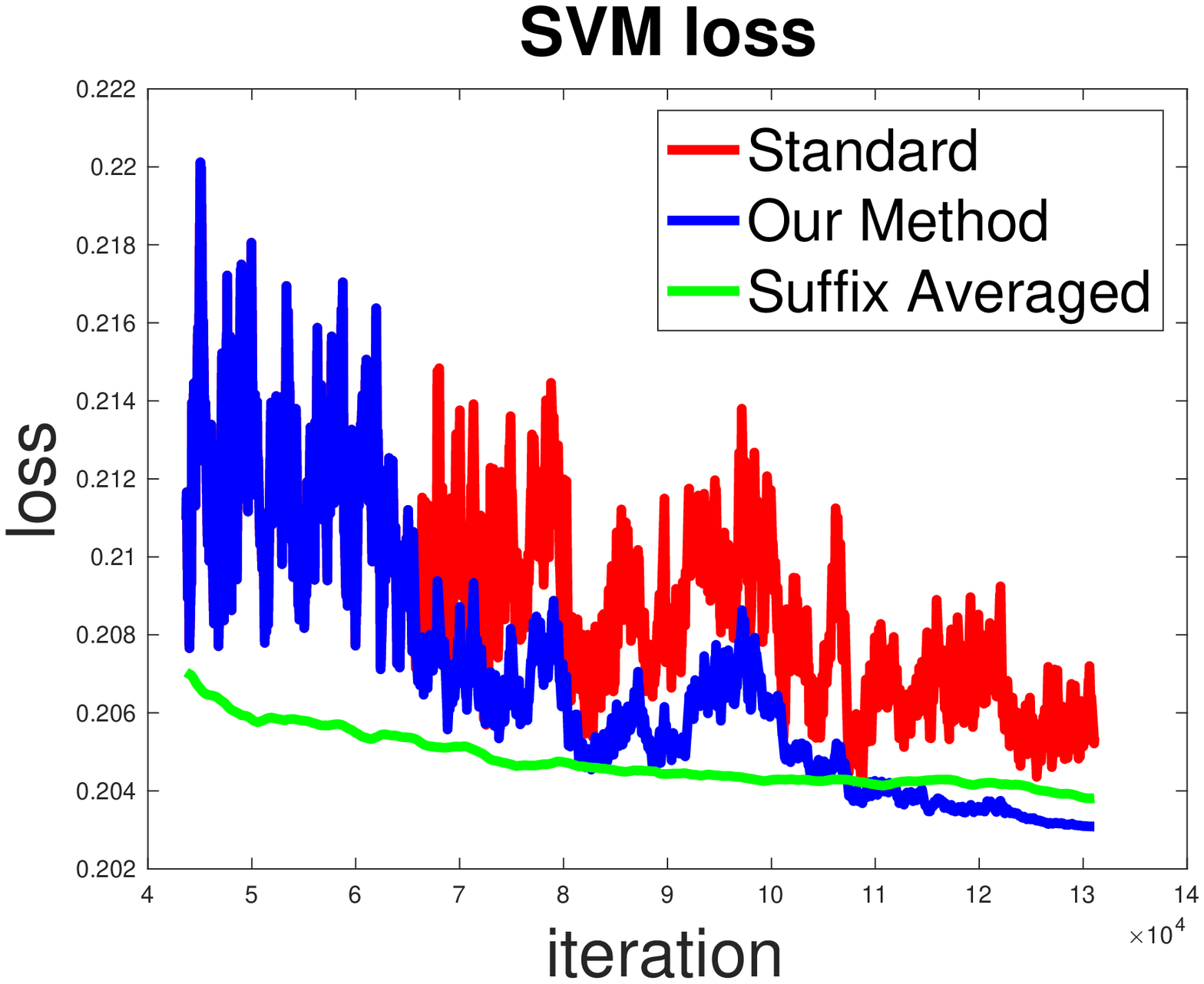}&
		\includegraphics[width=.32\textwidth]{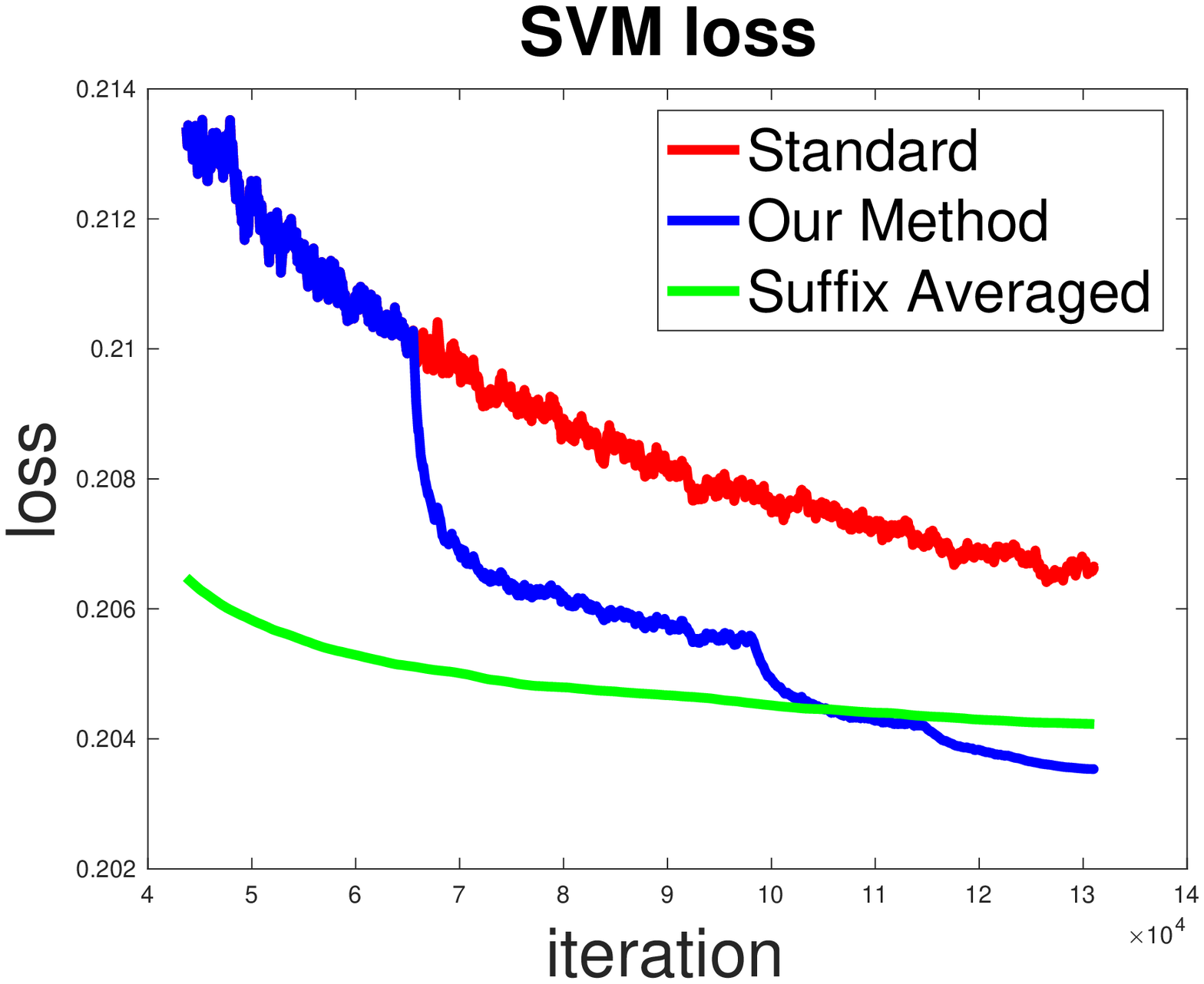}\vspace*{-5pt}\\
		(a)&(b)&(c)\vspace*{-10pt}
		\end{tabular}
		\caption{(a) $F(x)$ vs number of iterations for the Lasso regression (Section~\ref{sec:simulations}). Here  $d = 100$, $s=\|x^*\|_0=60$, $n=80$, $C=4$, $\sigma = 0.1$, $\lambda = 0.2$ and $T= 2^{19}$. Green line indicates the running average of the iterates from $1$ to $t$. (b) SVM loss \eqref{eq:svm} vs number of iterations. We pick $d = 30$, $\sigma = 5$, $\eta = 1$, $n=500$, $\lambda = 0.1$, $T = 2^{17}$. (c) Average over $100$ independent SGD runs for the SVM loss.  Green line is the loss of the average of the last $\frac{T}{4}$ iterates. In all the cases, SGD last point with our step-size sequence produces smaller objective value than the standard step size as well as the averaged iterates of SGD.}
		\label{fig:simple example}
	\vskip -0.1in
\end{figure}
Since the objective is not smooth, the gradient doesn't vanish near the optimum. Therefore, when the standard step size was picked, the iterate $x_t$ kept oscillating around the infimum but never really reaches it. In contrast, our method decreased the step size after sometime which allows better convergence to the optimum (see Figure~\ref{fig:simple example}(a)). 

\noindent{\bf Training SVMs:} We consider training SVMs which is a typical example where non-smooth SGD is heavily used \cite{shalev2011pegasos}. For our experiments, we generate data as follows. Let $a_i \sim \mathcal{N}(0,\sigma^2 I_d)$ and the label $b_i = \mathrm{sgn}(a_i(1)+z_i)$ where $z_i \sim \mathcal{N}(0,\eta^2)$. We generate $n=500$ points in $d=30$ dimensions. The SVM training problem is now: 
\begin{equation}\label{eq:svm}F(x) =   \frac{1}{n}\sum_{i=1}^{n}\max(0,1-b_i\langle a_i,x\rangle)+ \frac{\lambda}{2} \|x\|^2,\end{equation}
where $\lambda=0.1$. 
Since the objective is $\lambda$ strongly convex, we consider step sizes of $\gamma_t := \frac{1}{\lambda t}$ for the standard method and the modified step sizes given in Equation~\eqref{eq:strong_step_size} for our method.  Figure~\ref{fig:simple example} (b) plots loss during a typical run of SGD and Figure~\ref{fig:simple example} (c) for the loss averaged over $100$ independent runs of SGD for the same problem with the same initial point. The last point of SGD with modified step size sequence (Our Method)  in blue consistently outperformed the standard SGD (Standard) in red. The green line denotes the loss of the average of the last $\frac{T}{4}$ iterates.

%
%

\section{Conclusions and Discussion}\label{sec:disc}
We studied the fundamental question of sub-optimality of the last point of SGD/GD for general non-smooth convex functions as well as for strongly-convex functions. We proposed a novel step-size sequence that leads to information theoretically optimal rates in both the above mentioned settings. Our result proves a more general result for any ``modified step-size'' of a decaying standard step-size, and uses a novel technique of tracking best iterate in each time-interval and ensuring that the later iterates do not significantly deviate from the best iterate in the previous time interval. We also provide a high-probability bound using a super-martingale technique from~\cite{harvey2018tight}. Simulations show that our step-size indeed leads to better last point than the standard step-size sequences. 

Our approach fundamentally exploits an assumption that we apriori know the total number of iterations $T$. Hence, our result does not provide an any-time algorithm. In contrast, existing any-time results have an extra $\log T$ multiplicative factor in the sub-optimality. We conjecture that this gap is fundamental and  every {\em any-time} algorithm would suffer from the extra $\log T$ factor. We give lower bounds for the strongly convex case to show that for any choice of step sizes, the algorithm is either sub-optimal in expectation or almost surely so infinitely often. 
\section*{Acknowledgements}
This research was partially supported by ONR N00014-17-1-2147 and MIT-IBM Watson AI Lab.
\bibliographystyle{unsrt}  

\bibliography{paper}
\appendix
\section{Proofs of Technical Lemmas}
\label{sec:proofs_appendix} 
\subsection{Proof of Lemma~\ref{lem:look_ahead_high_probability}}
\begin{proof}[Proof of Lemma~\ref{lem:look_ahead_high_probability}]
We fix $l$ such that $t_0 \leq l\leq t_1$. In this proof, we will freely use the fact that $\alpha_{t_0} \geq \alpha_t$ whenever $t_0 \leq t$. Let $l \leq t\leq t_1$ Define $\Delta_t = \langle\hat{g}_t(x_{t}) -g_t(x_{t}),x_{t}-x_{l}\rangle $. We note that $x_t$ are random variables and are functions of $\hat{g}_1,\dots, \hat{g}_{t-1}$ only. We define the sigma-field $\mathcal{F}_t := \sigma(\hat{g}_1,\dots, \hat{g}_{t})$.  

We use the following notation for the sake of convenience: $D_t := \|x_t - x_{l}\|^2$. Clearly, $D_t$ is $\mathcal{F}_{t-1}$ measurable and $\Delta_t$ is $\mathcal{F}_{t}$ measurable. It is clear from the definition of $\Delta_t$ that $\mathbb{E}[\Delta_t|\mathcal{F}_{t-1}] = 0$ and $|\Delta_t| \leq 2 G\|x_t-x_{l}\| = 2G\sqrt{D_t}$. 

By Hoeffding's lemma, we conclude that for any $\mu \in \mathbb{R}$, we conclude:
\begin{equation}
\mathbb{E}\left[\exp(\mu \Delta_t)\bigr|\mathcal{F}_{t-1}\right] \leq \exp(2G^2D_t\mu^2) \label{eq:hoeffding_lemma}
\end{equation}

Let $\lambda =  \frac{1}{8 \alpha_{t_0}^2G^2}$. For $t_0 \leq t \leq t_1$, consider $$M_t := \exp\left( \sum_{s=l}^{t}-2\lambda L_s\alpha_s\Delta_s + \lambda(L_s-L_{s-1})D_{s}\right)\,.$$

Clearly, $M_t$ is $\mathcal{F}_t$ measurable. $M_{l} =1$ since $D_{l} = \Delta_{l} =0$ almost surely. We will show that $M_t$ is a super martingale:
\begin{align*}
\mathbb{E}\left[M_t\bigr|\mathcal{F}_{t-1}\right] &= M_{t-1}\mathbb{E}\left[\exp\left( -2\lambda L_t\alpha_t\Delta_t + \lambda(L_t-L_{t-1})D_{t}\right)\bigr|\mathcal{F}_{t-1}\right] \\
&= M_{t-1} \exp\left(-\lambda L_t^2D_t\right)\mathbb{E}\left[\exp\left( -2\lambda L_t\alpha_t\Delta_t \right)\bigr|\mathcal{F}_{t-1}\right] \\
&\leq M_{t-1} \exp\left(-\lambda L_t^2D_t + 8\lambda^2L_t^2\alpha_t^2G^2D_t\right)\\
&\leq M_{t-1}
\end{align*}

Therefore, \begin{equation}
\mathbb{E}\left[M_{t_1}\right] \leq 1
\label{eq:supermatringale_expectation}
\end{equation}
 From the proof of Lemma~\ref{lem:look_ahead}, for $l \leq t \leq t_1$ we have:
\begin{align}
\|x_{t+1}-x_{l}\|^2 &\leq \|x_t - x_{l}\|^2 + \alpha_t^2\|\hat{g}_t(x_t)\|^2 - 2\alpha_t \langle\hat{g}_t(x_{t}),x_{t}-x_{l}\rangle \nonumber   \\
&\leq  \|x_t - x_{l}\|^2 + \alpha_t^2G^2 -   2\alpha_t \langle g_t(x_{t}),x_{t}-x_{l}\rangle -2\alpha_t \langle\hat{g}_t(x_{t}) -g_t(x_{t}),x_{t}-x_{l}\rangle \nonumber  \\
&\leq  \|x_t - x_{l}\|^2  + \alpha_t^2G^2 +   2\alpha_t (F(x_{l}) - F(x_t))  -2\alpha_t \Delta_t \label{eq:distance_growth_pointwise}
\end{align}
In the third step, we have used the convexity of $F(\cdot)$. Reordering Equation~\eqref{eq:distance_growth_pointwise} and using the notation defined above:

$$2\alpha_t (F(x_{t}) - F(x_{l})) - \alpha_t^2G^2 \leq  D_t - D_{t+1} -2\alpha_t\Delta_t  \,. $$

Multiplying the equation above by $ L_t$ and adding from $t=l$ to $t = t_1$, noting the fact that $D_{l} =0$ and $D_{t_1 + 1} \geq 0$, we conclude:

\begin{align}
\sum_{t=l}^{t_1} L_t\left[2\alpha_t (F(x_{t}) - F(x_{l})) - \alpha_t^2G^2 \right] &\leq  \sum_{t=l}^{t_1} L_t\left(D_t - D_{t+1} -2\alpha_t\Delta_t  \right)\nonumber \\
&\leq  \sum_{t=l}^{t_1}  -2L_t\alpha_t\Delta_t  + (L_t -L_{t-1})D_t  \label{eq:supermartingale_domination}
\end{align}
 We recall the random variable $$A(l,t_1) := \sum_{t=l}^{t_1} L_t\left[2\alpha_t (F(x_{t}) - F(x_{l})) - \alpha_t^2G^2 \right]$$ 

From equations~\eqref{eq:supermatringale_expectation} and~\eqref{eq:supermartingale_domination}, we conclude that for every $l$ such that $t_0 \leq l\leq t_1$:

$$\mathbb{E}\left[\exp(\lambda A(l,t_1))\right] \leq \mathbb{E}\left[M_t\right] \leq 1\,. $$

By convexity of the exponential function, we have: 
\begin{align*}
\mathbb{E}\left[\exp(\lambda (p.A)(t_0,t_1))\right] &\leq \mathbb{E}\sum_{l=t_0}^{t_1}p(l)\mathbb{E}\left[\exp(\lambda A(l,t_1))\right] \leq 1
\end{align*}
By Chernoff Bound, we conclude:

$$\mathbb{P}\left[(p.A)(t_0,t_1) > \eta \right] \leq \exp\left(-\tfrac{\eta}{8\alpha^2_{t_0}G^2}\right)$$

The case for $A^*(t_0,t_1)$ proceeds similarly but this time we use $x^*$ in place of $x_{t_0}$. We define $D_t^* := \|x_t - x^*\|^2$, $\Delta_t^{*} :=  \langle\hat{g}_t(x_{t}) -g_t(x_{t}),x_{t}-x^{*}\rangle$ and $$M_t^{*} := \exp\left( \sum_{i=t_0}^{t}-2\lambda L_i\alpha_i\Delta^*_i + \lambda(L_i-L_{i-1})D^*_{i}\right)$$

We note that for $t_0<  t\leq t_1$, $\mathbb{E}\left[M^*_t \bigr| \mathcal{F}_{t-1}\right] \leq M^*_{t-1}$ and $D^*_{t_0} \leq D^2$. Therefore, 

$$\mathbb{E}M^*_{t_1} \leq \mathbb{E}M^*_1 \leq \exp\left( \tfrac{L^2_{t_0} D^2 +L_{t_0}D^2}{8\alpha^2_{t_0}G^2}\right) \leq \exp\left( \tfrac{2L_{t_0}D^2}{8\alpha^2_{t_0}G^2}\right)$$

Here we have used the fact that $L_{t_0} \leq \frac{1}{t_0-t_1+1} \leq 1$. Noting that $\exp(\lambda A^*(t_{0},t_1)) \leq M^*_{t_1}$  we use Chernoff bound to conclude the result.
\end{proof}

\subsection{Proof of Lemma~\ref{lem:lambda_growth_bound}}
\begin{proof}[Proof of Lemma~\ref{lem:lambda_growth_bound}]
We prove this by induction. The assertion is true for $i =0$. Suppose it is true for $ i =k \leq r-1$. Then, 
\begin{align*}
\lambda_{k+1} &= \lambda_{k} (1+\Gamma\lambda_k) \\
&\leq \left(1+\tfrac{1}{r}\right)^{k}\lambda_0\left(1+ \tfrac{\left(1+\tfrac{1}{r}\right)^k}{re}\right) \\
&\leq  \left(1+\tfrac{1}{r}\right)^{k}\lambda_0\left(1+ \tfrac{\left(1+\tfrac{1}{r}\right)^k}{re}\right) \\
&\leq \left(1+\tfrac{1}{r}\right)^{k+1}\lambda_0
\end{align*} 
The we have proved the assertion through induction.
\end{proof}

\subsection{Proof of Lemma~\ref{lem:sums_to_one}}
\begin{proof}[Proof of Lemma~\ref{lem:sums_to_one}]
We take the definitions of the terms used from the proof of Lemma~\ref{lem:good_point_transfer_high_probability}.
It is clear from the definition that $\kappa(t) \geq 0$. Since $\kappa(t) =0$ for $t > T_{i+1}$, it is sufficient to show that $\sum_{s=T_i+1}^{T_{i+1}} \kappa(s) = 1$.

We define $\sigma(t) = \sum_{s=T_i+1}^{t}\kappa(s)$ (an empty sum denotes 0). By definition of $\kappa$, for $ T_i +1\leq t \leq T_{i+1}$ 
\begin{align*}
\sigma(t) &= \frac{\Gamma(T_{i+1})}{\Gamma(t)}q^{(i)}(t) + \left(1+\frac{\alpha_t L_t}{\Gamma(t)}\right)\sigma(t-1) \\
&= \frac{\Gamma(T_{i+1})}{\Gamma(t)}q^{(i)}(t) + \frac{\Gamma(t-1)}{\Gamma(t)}\sigma(t-1)
\end{align*}
Continuing the above recursion, we conclude:
$$\sigma(t) =  \frac{\Gamma(T_{i+1})}{\Gamma(t)} \sum_{s=T_i+1}^{T}q^{(i)}(t)$$

Since $q$ is a probability distribution over $\{T_{i}+1,\dots,T_i\}$, we conclude $$\sigma(T_{i+1}) = \sum_{s= T_i+1}^{T_{i+1}}q^{(i)}(t) = 1$$.

\end{proof}

\section{Proofs of Lower Bounds}
\label{sec:lower_bounds}

We will prove theorem~\ref{thm:lower_bounds} for $G=5$ and $\mu = 1$ for the sake of convenience. We can handle the general case by considering the transformation $F_0(x) = \frac{25\mu}{G^2}F(\frac{G}{2\mu}x)$. We scale the domain as $D_0 := \frac{5\mu}{G}D$.  If $F$ is $\mu$ strongly convex and $G$ Lipschitz, then $F_0$ is $1$ strongly convex and $5$ Lipschitz. We take the subgradient oracle for $F_0$ to be $\hat{g}^{0}_t(x) :=\frac{5}{G}\hat{g}_t(\frac{Gx}{5\mu})$. It is easy to check that if SGD for $F(.)$ with step sizes $\alpha_t$, the iterates are $x_t$, then starting from $x^0_0 := \frac{5\mu}{G}x_0$ and using step sizes $\alpha^{0}_t := \mu \alpha_t$ and the subgradient oracle defined above, the iterates for $F_0$ is $x_t^{0} = \frac{5\mu}{G}x_t$. Therefore, $F_0(x_t^{0}) = \frac{25\mu}{G^2}F(x_t)$ and the proof below goes through seamlessly. This is similar to the rescaling used for the lowerbounds in \cite{harvey2018tight}. 

Without loss of generality, we will restrict our attention to strictly positive step size sequences: $\gamma_t >0$. We further restrict the possible values of $\gamma_t$ in the following lemma:

\begin{lemma}
\label{lem:sub_harmonic_steps}
If the step size sequence $\gamma_t$ is such that there is an infinite sequence of times $t_k$ such that $\lim_{k\to \infty}t_k\gamma_{t_k} = \infty$, then SGD is bad in expectation. Therefore, we can restrict our consideration to step size sequences of the form $\gamma_t = O(\frac{1}{t})$.
\end{lemma}

\begin{proof}
Consider the function $F : [-1,1] \to \mathbb{R}$ defined by $F(x) = |x| + \frac{x^2}{2}$. $F$ has a global optimum at $x =0$ and it is $1$ strongly convex. Let $\epsilon_t$ be a sequence of i.i.d. rademacher random variables (i.e, uniform over $\{-1,1\}$). We let the subgradient oracle to return $\hat{g}_t(x) = \mathsf{sgn}(x) + x + 3\epsilon_t$. Clearly,
$$|x_{t+1}| =\min( |x_t - \gamma_t(\mathsf{sgn}(x_t) + x_t + 3\epsilon_t)|,1)\,.$$
$\epsilon_t$ is independent of $x_t$ and conditioned on the value of $x_t$, with probability atleast $\frac{1}{2}$, 
$\epsilon_t$ has the opposite sign as $ x_t $. When this happens, $(\mathsf{sgn}(x_t) + x_t + 3\epsilon_t)$ has the opposite sign of $x_t$ and $|(\mathsf{sgn}(x_t) + x_t + 3\epsilon_t)| \geq 1$. Therefore under this event, $|x_t - \gamma_t(\mathsf{sgn}(x_t) + x_t + 3\epsilon_t)| \geq |x_t| + \gamma_t \geq \gamma_t$.

Therefore, we conclude:
$$\mathbb{E}|x_{t+1}| \geq \frac{1}{2}\min(1,\gamma_t)\,.$$ Considering the fact that $(t_k+1)\mathbb{E}\left(F(x_{t_k+1})- F(0)\right)\geq t_{k}\mathbb{E}|x_{t_k+1}| \geq \frac{1}{2}\min(t_k,\gamma_{t_k}t_k) \to \infty$, we conclude that SGD with this step size is bad in expectation.

\end{proof}

Henceforth, we will restrict our attention without loss of generality to step size sequences such that $\gamma_t = O(\frac{1}{t})$.
We will first consider the function $F_1(x) = \frac{1}{2}x^2$ over the set $[-1,1]$. Let the infinite horizon learning rate be $(\gamma_t)_{t \in \mathbb{N}}$ at each time instant, the subgradient oracle returns $x + \epsilon_t$ where $\epsilon_t$ is a sequence of i.i.d. uniform random variable over $\{-1,1\}$ (that is rademacher random variables). Let the iterates of SGD be $z_t$ and $z_1 = 1$.

\begin{lemma}
Let $T_0$ be the smallest time such that $\gamma_t < 1$ for all $t \geq T_0$. Then, for every $t \geq T_0$
\begin{enumerate}
\item 
$$\mathbb{E}|z_{t+1}|^2 = (1-\gamma_t)^2 \mathbb{E} z_t^2 + \gamma_t^2$$
\item
$$\mathbb{E}\|z_t\|^2 \geq \frac{1}{t-T_0+1}$$
\end{enumerate}

\label{lem:square_example}
\end{lemma} 

\begin{proof}
Suppose $\gamma_t \geq 1$. Then:
\begin{align*}
|z_t (1-\gamma_t) + \gamma_t \epsilon_t| &\geq \gamma_t|\epsilon_t| - |z_t(1-\gamma_t)| 
\geq \gamma_t - |1-\gamma_t|  =  1
\end{align*} 

Therefore, when $\gamma_t \geq 1$, $|z_{t+1}| = 1$. Therefore, $z_{T_0} = 1$ almost surely. When $t\geq T_0$, 
\begin{align*}
|z_t - \gamma_t(z_t + \epsilon_t)| &\leq |(1-\gamma_t)z_t| + \gamma_t |\epsilon_t| \\
&= (1-\gamma_t)|z_t| + \gamma_t \leq 1  
\end{align*}
Therefore, when $t \geq T_0$, the iteration of SGD won't leave the set $[-1,1]$ almost surely, so there is no need for the projection step to obtain the next iterate. That is, for $t \geq T_0$, $z_{t+1} = z_t(1-\gamma_t) + \epsilon_t \gamma_t$. Squaring and taking expectations, we conclude:

\begin{align*}
\mathbb{E}|z_{t+1}|^2 &= (1-\gamma_t)^2 \mathbb{E} z_t^2 + \gamma_t^2 \\
&\geq \inf_{\gamma\in \mathbb{R}} (1-\gamma)^2 \mathbb{E} z_t^2 + \gamma^2 \\
&= \frac{\mathbb{E} z_t^2}{1+ \mathbb{E} z_t^2}
\end{align*}

Clearly, $\mathbb{E}|z_{T_0}|^2 = 1 = \frac{1}{T_0-T_0+1}$. Using induction in the equation above, we conclude:
 $\mathbb{E}|z_{t}|^2 \geq \frac{1}{t - T_0 +1}$ for every $t \geq T_0$.

\end{proof}

We divide $\mathbb{N}$ into time intervals of the form $\{2^{k},2^{k}+1,\dots, 2^{k+1}-1\} := I_k$. We have the following lemma:

\begin{lemma}
If $\gamma_t =\leq \frac{C}{t}$ for some constant $C \geq 0$ and there exist positive infinite sequences $c_k$ and $d_k$ such that $\lim_{k\to \infty}c_k = \infty$, $\lim_{k\to \infty}d_k = 0$ and every $k$, either one of the two conditions below hold:
\begin{enumerate}
\item
$\sum_{t\in I_k}\gamma_t^2 \geq c_k 2^{-k} \left(\sum_{t\in I_k}\gamma_t\right)^2$
\item $\sum_{i\in I_k}\gamma_t \leq d_k$
\end{enumerate}
Then, SGD with step size $\gamma_t$ is bad in expectation.
\label{lem:high_var_bad_expectation}
\end{lemma}

\begin{proof}
We consider the optimization problem considered in Lemma~\ref{lem:square_example} i.e, optimizing $F(x) = x^2$. Let $T_k = 2^k$. We assume the contrary - that is, $\mathbb{E}|z_{T_k}|^2 \leq \frac{L}{T_k}$ for every $k$, for some $L >0$. As shown in the second inequality of Lemma~\ref{lem:square_example}, irrespective of the choice of $\gamma_1,\dots, \gamma_{T_k-1}$, $$\mathbb{E}|z_{T_k}^2| \geq \frac{1}{T_k-1}\,.$$
From the first equality in Lemma~\ref{lem:square_example}, we conclude that for $t \in I_k$, $\mathbb{E}|z_{t+1}|^2 = (1-\gamma_t)^2 \mathbb{E} z_t^2 + \gamma_t^2$
Since $\gamma_t \leq \frac{C}{t}$, we can take $k$ large enough so that $\gamma_t \leq \frac{1}{2}$ for every $t \in I_k$. Using the fact that $(1-\gamma_t) \geq \exp{-\frac{\gamma_t}{1-\gamma_t}} \geq \exp{-2\gamma_t}$. Therefore, 

$$\mathbb{E}|z_{t+1}|^2 \geq e^{-4\gamma_t} \mathbb{E}|z_{t}|^2 + \gamma_t^2$$
Unravelling the recursion above, we conclude:

\begin{equation}
\mathbb{E}|z_{T_{k+1}}|^2 \geq e^{-4\sum_{t\in I_k}\gamma_t}\left[ \mathbb{E}|z_{T_k}|^2 + \sum_{t \in I_k} \gamma_t^2\right] \label{eq:inter_window_contraction}
\end{equation}

We define $S_k := \sum_{t\in I_k} \gamma_t$. 

\begin{enumerate}

\item Suppose for a particular $k$, the first item in the statement of the lemma holds

 By assumption, $\sum_{t \in I_k} \gamma_t^2 \geq \frac{c_k}{T_k}S_k^2$. Using this in Equation~\eqref{eq:inter_window_contraction}, we conclude:

$$\mathbb{E}|z_{T_{k+1}}|^2 \geq e^{-4S_k}\left[ \mathbb{E}|z_{T_k}|^2 +\frac{c_k}{T_k}S_k^2 \right]$$

Now, since $\gamma_t \leq \frac{C}{t}$, we have $S_k \leq C$. Therefore, 

\begin{align}
\mathbb{E}|z_{T_{k+1}}|^2 &\geq e^{-4S_k} \mathbb{E}|z_{T_k}|^2 +\frac{c_k}{T_k}S_k^2 e^{-4C} \nonumber \\
&=  \mathbb{E}|z_{T_k}|^2 \left[e^{-4S_k} +\frac{c_k S_k^2 e^{-4C}}{T_k\mathbb{E}|z_{T_k}|^2} \right] \nonumber\\
&\geq  \mathbb{E}|z_{T_k}|^2 \left[e^{-4S_k} +\frac{c_k S_k^2 e^{-4C}}{L} \right] \nonumber\\
&\geq \mathbb{E}|z_{T_k}|^2 \inf_{x \geq 0} \left[e^{-4x} + \frac{x^2c_k e^{-4C}}{L}\right] \label{eq:bad_contraction_1}
\end{align}

In the third step, we have used the fact that $T_k \mathbb{E}|z_{T_k}|^2 \leq L$.  We now consider the function $h : \mathbb{R}^{+} \to \mathbb{R}^{+}$ given by $h(x) = e^{-2x} + \kappa x^2$ for some $\kappa >0$. Clearly,  $h$ is convex, bounded below and tends to infinity as $x\to \infty$. Therefore, it has a unique minimizer $t^*$ - the unique point such that $h^{\prime}(t^{*}) =0$. That is, $t^{*}$ is the unique point which satisfies: $\kappa t^{*} = 2 e^{-4t^{*}} \leq 2$. Therefore, $t^{*} \leq \frac{2}{\kappa}$. Therefore, $h(t^{*}) \geq e^{-4t^{*}} \geq e^{-8/\kappa} \geq 1- \frac{8}{\kappa}$.  In Equation~\eqref{eq:bad_contraction}, we take $\kappa = \frac{c_ke^{-4C}}{L}$ we conclude:

\begin{align*}
\mathbb{E}|z_{T_{k+1}}|^2 \geq \mathbb{E}|z_{T_k}|^2 \left(1- \tfrac{C^{\prime}}{c_k}\right)
\end{align*}

 Where $C^{\prime}$ is a constant depending only on $L$ and $C$.
\item  Suppose for a particular $k$, the second item in the statement of the lemma holds:
Then, by Equation~\eqref{eq:inter_window_contraction}, we have:

\begin{equation}
\mathbb{E}|z_{T_{k+1}}|^2 \geq e^{-4d_k}\mathbb{E}|z_{T_k}|^2 \geq (1-4d_k)\mathbb{E}|z_{T_k}|^2
\label{eq:bad_contraction_2}
\end{equation}

\end{enumerate}

From Equations~\eqref{eq:bad_contraction_1} and ~\eqref{eq:bad_contraction_2}, we conclude that there exists an absolute constant $\bar{C}$ depending only on $C$ and $L$ such that:

\begin{equation}
\mathbb{E}|z_{T_{k+1}}|^2 \geq \left(1-\bar{C}\max\left(d_k,\tfrac{1}{c_k}\right)\right) \mathbb{E}|z_{T_k}|^2 
\label{eq:bad_contraction}
\end{equation}
 
Since $\max\left(d_k,\tfrac{1}{c_k}\right) \to 0$, we can choose $k$ large enough so that $\sup_{s >k}\bar{C}\max\left(d_k,\tfrac{1}{c_k}\right) \leq 1-e^{-\epsilon}$ for arbitrary $\epsilon >0$. 
 
From Equation~\eqref{eq:bad_contraction}, it follows that for arbitrary $K \in \mathbb{N}$,
 $$\mathbb{E}|z_{T_{k+K}}|^2 \geq e^{-\epsilon K} \mathbb{E}|z_{T_k}|^2 $$

By Lemma~\ref{lem:square_example}, $\mathbb{E}|z_{T_k}|^2  \geq \frac{1}{T_k-1} \geq 2^{-k}$. By our assumption, $\mathbb{E}|z_{T_{k+K}}|^2 \leq L 2^{-k-K}$. Therefore, we conclude:$L 2^{-k-K} \geq e^{-\epsilon K} 2^{-k}$ for every $K \in \mathbb{N}$. This cannot hold for any finite $L$ when we take $\epsilon < \log{2}$. This contradicts our assumption. Therefore, SGD with step size $\gamma_t$ is bad in expectation.

\end{proof}

We will show that if conditions for $\gamma_t$ in Lemma~\ref{lem:square_example} or those in Lemma~\ref{lem:high_var_bad_expectation} don't hold, then SGD is bad almost surely. We recall the definition of the interval $I_k = \{2^{k}+1,\dots, 2^{k+1}\}$. We prove the following lemma to inspect how frequently long, contiguous segments of $\epsilon_t$ are all equal to $1$ for $t \in I_k$. We take $\tau_k := 2^{\lfloor\log_2(k/2) \rfloor}$.  We note that $\frac{k}{4} \leq \tau_k \leq \frac{k}{2}$We can divide $I_k$ into $|I_k|/\tau_k$ contiguous, disjoint intervals, each of size $\tau_k$. We call these intervals $J_k(i)$ for $i\in\{1,\dots, |I_k|/\tau_k\}$. We let $A_k$ to be the event that for some $i \in\{1,\dots, |I_k|/tau\} $, $\epsilon_t = 1$ for all $t \in J_k(i)$. In particular, the even $A_k$ implies that there is a contiguous $\tau_k$ length sequence  of $\epsilon_t$ of all $1$s in $I_k$.

\begin{lemma}

\item $\mathbb{P}(A_k^{c}) \leq C k 2^{-k/2}$ for some absolute constant $C$.

\label{lem:infinitely_often}
\end{lemma}
\begin{proof}
We subdivide the interval $I_k$ into disjoint subintervals of length $\tau_k$. There are $\tfrac{2^k}{\tau_k} $ such intervals. The event $A_k$ holds if over one such subinterval, the random signs are all $1$. The probability of a given subinterval having all signs equal to $1$ is $p_{\tau_k}:= \tfrac{1}{2^{\tau_k}}$. Therefore, we conclude:

$$\mathbb{P}(A_{k}) = 1- (1-p_{\tau})^{\lfloor\tfrac{2^k}{\tau} \rfloor}\geq 1- e^{-p_{\tau}\lfloor\tfrac{2^k}{\tau} \rfloor} \geq 1- \frac{1}{ep_{\tau_k}\lfloor\tfrac{2^k}{\tau_k} \rfloor}\,.$$

Here we have used the inequality $xe^{-x} \leq \frac{1}{e}$ for $x > 0$. 

Therefore, we conclude that: $\mathbb{P}(A_k^{c}) \leq Ck2^{-k/2}$ for some absolute constant $C$. 
\end{proof}

We now consider the same function which was considered in Lemma~\ref{lem:sub_harmonic_steps} i.e, $F: [-1,1] \to \mathbb{R}$ defined by $F(x) = |x| + \frac{x^2}{2}$.  $F$ has a global optimum at $x =0$ and it is $1$ strongly convex. Let $\epsilon_t$ be a sequence of i.i.d. rademacher random variables (i.e, uniform over $\{-1,1\}$). We let the subgradient oracle to return $\hat{g}_t(x) = \mathsf{sgn}(x) + x + 3\epsilon_t$.  Let the iterates of SGD for $F$ with step sizes $\gamma_t$ be $y_t$.

\begin{lemma}
Suppose $\gamma_t \leq \frac{C}{t}$, there exists an infinite sequence $(k_r)_{r\in \mathbb{N}}$ and fixed constants $c_0,d_0 >0$ such that both the conditions hold:
\begin{enumerate}
\item  $$\sum_{t\in I_{k_r}}\gamma_t^2 \leq c_0 2^{-k_r} \left(\sum_{t\in I_{k_r}}\gamma_t\right)^2$$  
\item $$\sum_{t \in I_{k_r}} \gamma_t \geq d_0$$
\end{enumerate}
We note that these conditions are the negations of the conditions for $\gamma_t$ in Lemma~\ref{lem:sub_harmonic_steps} and Lemma~\ref{lem:high_var_bad_expectation}. Then SGD with step size $\gamma_t$ is bad almost surely. 
\label{lem:bad_almost_surely}
\end{lemma}
\begin{proof}
We will show that there exists a sequence of independent events $B_{k_r}$ for $r \in \mathbb{N}$ such that $\mathbb{P}(B_{k_r}) \geq p_0 > 0$ uniformly and whenever $B_{k_r}$ holds, 

$$\max_{t \in I_{k_r}} \frac{t}{\log t}\left[ F(y_t) - F(0)\right] \geq \delta_0$$
For some constant $\delta_0 >0$. We note that $p_0$ and $\delta_0$ depend only on $C,d_0$ and $c_0$. We consider a random times $T_{\mathsf{max}},T_{\mathsf{min}} \in I_k$ as follows:

\begin{enumerate}
\item If the event $A_k^{c}$ holds, pick a uniformly random element $i_0$ from $\{1,\dots, |I_k|/\tau_k\}$ independent of everything else. Set $T_{\mathsf{max}} := \max J_k(i_0)$ and $T_{\mathsf{min}} := \min J_k(i_0)$
\item If the event $A_k$ holds, pick a uniformly random element $i_0$ from $\{i: \text{ for all } t \in J_k(i), \epsilon_t = 1\}$, independent of everything else. Set $T_{\mathsf{max}} := \max J_k(i_0)$ and $T_{\mathsf{min}} := \min J_k(i_0)$
\end{enumerate}

We note that by symmetry, $i_0$ is uniformly distributed over the set $\{1,\dots, |I_k|/\tau_k\}$. We will show that, when the event $A_k$ holds, then one of the following is true:
\begin{enumerate}
\item $y\left(T_{\mathsf{max}}\right) = -1$.
\item $y\left(T_{\mathsf{min}}\right) - y\left(T_{\mathsf{max}}\right) \geq \sum_{t \in J_k(i_0)}\gamma_t$
\end{enumerate}

Suppose the event $A_k$ holds. Then for $T_{\mathsf{min}} \leq t < T_{\mathsf{max}}$, $y_{t+1} = \max(y_t - \gamma_t (y_t + \mathsf{sgn}(y_t) + 3),-1)$. Since under the event $A_k$, $\epsilon_t = 1$ for every $t\in J_k(i_0)$, we conclude that $\gamma_t (y_t + \mathsf{sgn}(y_t) + 3) \geq \gamma_t$. That is SGD drifts in the negative direction irrespective of the value of the iterate. It is therefore clear that if for some $T_{\mathsf{min}}\leq t \leq T_{\mathsf{max}}$, $y_t$ hits $-1$, then $y\left(T_{\mathsf{max}}\right)  = -1$. Now suppose that for $y_t >-1$ for every $t$ in this range. Then, $y_{t+1} \leq y_{t} - \gamma_t$. But unraveling this recursion, it follows that $y\left(T_{\mathsf{min}}\right) - y\left(T_{\mathsf{max}}\right) \geq \sum_{t \in J_k(i_0)}\gamma_t$. Therefore, it follows that when the event $A_k$ holds:

\begin{align}
\max_{t \in I_k} F(y_t) &\geq \max(F(y(T_{\max})), F(y(T_{\min})))\nonumber \\
&\geq \max(|y(T_{\max})|, |y(T_{\min})|)\nonumber\\
&\geq \min\bigr(1, \tfrac{1}{2}\sum_{t \in J_k(i_0)}\gamma_t\bigr) \label{eq:drift_maxima}
\end{align}

It is clear that since $\gamma_t = C/t$, for $k$ large enough,  $\tfrac{1}{2}\sum_{t \in J_k(i_0)}\gamma_t \leq 1$. Therefore, we conclude that for $k$ large enough, when the event $A_k$ holds,

$$\max_{t \in I_k} F(y_t)  \geq  \tfrac{1}{2}\sum_{t \in J_k(i_0)}\gamma_t \,.$$

Fix $0<\beta <1$.We now consider $E_k$ to be the event  $\bigr\{\sum_{t \in J_k(i_0)}\gamma_t \geq \frac{\beta\tau}{|I_k|} \sum_{t\in I_k}\gamma_t\bigr\}$.

By symmetry, $i_0$ is uniformly distributed over $\{1,\dots, |I_k|/\tau_k\}$. Therefore, $$\mathbb{E}\sum_{t \in J_k(i_0)}\gamma_t = \frac{\tau_k}{|I_k|} \sum_{t\in I_k}\gamma_t$$ and 
\begin{align*}
\mathbb{E}\left(\sum_{t \in J_k(i_0)}\gamma_t\right)^2 &= \frac{\tau_k}{|I_k|}\sum_{i=1}^{|I_k|/\tau_k} \sum_{t,s \in J_k(i_0)} \gamma_t\gamma_s \\
&\leq  \frac{\tau_k}{|I_k|}\sum_{i=1}^{|I_k|/\tau_k} \sum_{t,s \in J_k(i_0)} \frac{\gamma_t^2+\gamma_s^2}{2} \\
&= \frac{\tau_k^2}{|I_k|} \sum_{t\in I_k} \gamma_t^2 
\end{align*}

Now, when $k$ is part of the infinite sequence $(k_r)$, by assumption we have:
\begin{align*}
\mathbb{E}\left(\sum_{t \in J_k(i_0)}\gamma_t\right)^2 &\leq c_0 \frac{\tau_k^2}{|I_k|^2}\left( \sum_{t\in I_k}\gamma_t\right)^2
\end{align*}

Therefore, by Payley-Zigmund inequality, whenever $k$ is part of the infinite sequence $(k_r)$,  for every $\beta < 1$,

$$\mathbb{P}\left[\sum_{t \in J_k(i_0)}\gamma_t \geq \beta \frac{\tau}{|I_k|} \sum_{t\in I_k}\gamma_t\right] \geq (1-\beta)^2\frac{ (\mathbb{E}\sum_{t \in J_k(i_0)}\gamma_t )^2}{\mathbb{E}(\sum_{t \in J_k(i_0)}\gamma_t )^2}
\geq \frac{(1-\beta)^2}{c_0} 
$$

Recalling the definition of $E_k$, we conclude, $\mathbb{P}(E_{k_r}) \geq  \frac{(1-\beta)^2}{c_0} $.

We will now define the event $B_k := E_k\cap A_k$.  The events $B_k$ are all independent by definition. When the event $B_k$ holds, clearly, from equation~\eqref{eq:drift_maxima}, we conclude:

$$\max_{t \in I_{k_r}} F(y_t)  \geq  \tfrac{1}{2}\sum_{t \in J_{k_r}(i_0)}\gamma_t \geq\frac{\beta\tau_{k_r}}{2|I_{k_r}|} \sum_{t\in I_{k_r}}\gamma_t \geq \frac{\beta\tau_{k_r} d_0}{2|I_{k_r}|}$$
The second inequality follows from the defintion of $E_k$. Using the fact that any $t \in I_{k_r}$ is such that $t \leq 2|I_{k_r}|$ and $\tau_{k_r}= \Theta(k_r)$, we conclude that for some $\delta_0 >0$, fixed, the following holds whenever the event $B_{k_r}$ holds.

\begin{equation}
\max_{t \in I_{k_r}} \frac{t}{\log t}\left[ F(y_t) - F(0)\right] \geq \delta_0
\label{eq:log_deviation_io}
\end{equation}

\begin{align*}
\mathbb{P}(B_{k_r}) &\geq \mathbb{P}(E_{k_r}) - \mathbb{P}(A_{k_r})\\
& \geq \frac{(1-\beta)^2}{c_0} - O(k_r 2^{-k_r/2})
\end{align*}

It is clear that we can find a $p_0 >0$ such that for all $k_r$ large enough, $\mathbb{P}(B_{k_r}) > p_0$. 

Since $B_{k_r}$ are independent sets, it follows that infinitely many of them are true with probability $1$. From equation~\eqref{eq:log_deviation_io}, we conclude that SGD with step sizes $\gamma_t$ is bad almost surely.
\end{proof}

\begin{proof}[Proof of Theorem~\ref{thm:lower_bounds}]
We will conclude this from Lemmas~\ref{lem:sub_harmonic_steps}, ~\ref{lem:high_var_bad_expectation} and~\ref{lem:bad_almost_surely}. Therefore, it is sufficient to show that any strictly positive infinite sequence $\gamma_t$ is such that atleast one of the following condition holds
\begin{enumerate}
\item There is an infinite sequence of times $t_k$ such that $\lim_{k\to \infty}t_k\gamma_{t_k} = \infty$. In this case, by Lemma~\ref{lem:sub_harmonic_steps}, we conclude that it is bad in expectation.
\item There exists a $C$ such that $\gamma_t \leq \frac{C}{t}$ and there exist infinite sequences $c_k \to \infty$ and $d_k \to 0$ such that for every k, either
 $\sum_{t\in I_k}\gamma_t^2 \geq c_k 2^{-k} \left(\sum_{t\in I_k}\gamma_t\right)^2$ or $\sum_{t \in I_{k}} \gamma_t \leq d_k$. In this case, by Lemma~\ref{lem:high_var_bad_expectation}, we conclude that it is bad in expectation.
\item There exists a $C$ such that $\gamma_t \leq \frac{C}{t}$ and there exist fixed positive constants $c_0$ and $d_0$ such that for some infinite sub-sequence $(k_r)$,
$\sum_{t\in I_{k_r}}\gamma_t^2 \leq c_0 2^{-k_r} \left(\sum_{t\in I_{k_r}}\gamma_t\right)^2$ and $\sum_{t \in I_{k_r}} \gamma_t \geq d_0$. In this case, by Lemma~\ref{lem:bad_almost_surely}, we conclude that the algorithm is bad almost surely.
\end{enumerate}

It is therefore sufficient to show that if conditions 1 and 2 don't hold then condition 3 holds. The negation of condition 1 is that $\gamma_t \leq \frac{C}{t}$ for some $C >0$. Now, we denote by $$\eta_k := 2^k\frac{\sum_{t\in I_k}\gamma_t^2}{\left(\sum_{t\in I_k}\gamma_t\right)^2}$$ and $$\lambda_k := \sum_{t \in I_{k}} \gamma_t $$. Therefore, $\eta_k \geq c_k$ or $\lambda_k \leq  d_k$ for some $c_k \to \infty$ and $d_k \to 0$ is equivalent to $\eta_k + \frac{1}{\lambda_k} \to \infty$ which is equivalent to the statement that for every subsequence $k_r$,  $\eta_{k_r} + \frac{1}{\lambda_{k_r}} \to \infty$. Therefore the negation of condition 2 is equivalent to atleast one of the following conditions being true
\begin{enumerate}
\item There exists infinite sequence $(t_k)$ such that $t_k\gamma_{t_k} \to \infty$
\item There exists and infinite subsequence $k_r$ such that $\eta_{k_r} + \frac{1}{\lambda_{k_r}} \leq M$ for some $M >0$. That is,  $\eta_{k_r} \leq M := c_0$ and $\lambda_{k_r} \geq \frac{1}{M} := d_0$ 
\end{enumerate} 
Therefore we conclude that when neither of the conditions 1 and 2 hold, then condition 3 holds. This proves our result.

\end{proof}
\end{document}